\documentclass{article}
\usepackage{amsfonts}
\usepackage{amsmath}
\usepackage{amssymb}
\usepackage{graphicx}%
\providecommand{\U}[1]{\protect\rule{.1in}{.1in}}
\numberwithin{equation}{section}%
\numberwithin{figure}{section}%%

\newcommand{\ignore}[1]{}
\newtheorem{theorem}{Theorem}

\newtheorem{lemma}[theorem]{Lemma}

\newtheorem{remark}[theorem]{Remark}

\newenvironment{proof}[1][Proof]{\noindent\textbf{#1.} }{\ \rule{0.5em}{0.5em}}
\begin{document}

\title{
Asymptotics of self-similar solutions to
coagulation equations with product kernel}
\author{J.B. McLeod \footnote{Mathematical Institute,
 University of Oxford, 24--29 St Giles', Oxford OX13LB, United Kingdom. {\tt E-mail}: mcleod@maths.ox.ac.uk}, Barbara Niethammer \footnote{Mathematical Institute, University of Oxford, 24--29 St Giles', Oxford OX13LB, United Kingdom. {\tt E-mail}: niethammer@maths.ox.ac.uk}
and J.J.L. Vel\'azquez \footnote{ICMAT,
 C/ Nicolas Cabrera 15, Cantoblanco, 28049 Madrid. {\tt E-mail}:jj$\mbox{}_{-}$velazquez@icmat.es}
}
\date{}
\maketitle

\begin{abstract}
We consider mass-conserving self-similar solutions for Smoluchowski's coagulation equation
with kernel $K(\xi,\eta)= (\xi \eta)^{\lambda}$ with $\lambda \in (0,1/2)$.
It is known that such self-similar solutions $g(x)$ satisfy that $x^{-1+2\lambda} g(x)$ is bounded
above and below as $x \to 0$. In this paper we describe in detail
 via formal asymptotics  
the qualitative
behavior of a suitably
 rescaled function $h(x)=h_{\lambda} x^{-1+2\lambda} g(x)$ in the limit $\lambda \to 0$.
It turns out that  $h \sim 1+ C x^{\lambda/2} \cos(\sqrt{\lambda} \log x)$ 
as $x \to 0$.  As $x$ becomes larger
 $h$ develops peaks of height $1/\lambda$ that are separated
by large regions where $h$ is small. Finally, $h$ converges to zero exponentially fast as
$x \to \infty$. Our analysis is based on different approximations of a nonlocal operator, that
reduces the original equation in certain regimes to a system of ODE. 

\end{abstract}

\section{Introduction}

\bigskip In this article we consider self-similar solutions to Smoluchowski's
mean-field model for coagulation \cite{Drake,Smolu} that is given by the
equation
\begin{equation}%
\begin{split}
\partial_{t}f\left(  \xi,t\right)   &  =\frac{1}{2}\int_{0}^{\xi}K\left(
\xi-\eta,\eta\right)  f\left(  \xi-\eta,t\right)  f\left(  \eta,t\right)  d\eta\\
& \quad-\int_{0}^{\infty}K\left(  \xi,\eta\right)  f\left(  \xi,t\right)
f\left(  \eta,t\right)  d\eta=:Q\left[  f\right]\,, \label{S1E1}%
\end{split}
\end{equation}
where $f(\xi,t)$ denotes the number density of clusters of size $\xi$ at time
$t$. The kernel $K(\xi,\eta)$ describes the rate of coalescence of clusters of
size $\xi$ and $\eta$ and subsumes all the microscopic aspects of the
coagulation process. It is well-known, cf. for example the review article
\cite{LM1} and references therein,  that if the kernel grows at most linearly,
then the initial value problem for (\ref{S1E1}) is well-posed for initial data
with finite mass and the mass is conserved for all times. For homogeneous
kernels it is furthermore expected that solutions converge to a self-similar
form for large times. In fact, for  the kernels $K(\xi,\eta) \equiv1$ and
$K(\xi,\eta)=\xi+\eta$, the constant and additive kernel respectively, this
issue has by now been completely solved. It has been established
\cite{Bertoin1,MP1} that besides explicitly known exponentially decaying
self-similar solutions also solutions with algebraic decay exist. In
\cite{MP1} their domains of attraction could also be completely characterized.
However, not much is known about self-similar solutions for other kernels than
these solvable ones. Existence of fast-decaying self-similar solutions has been
established for a large class of kernels \cite{FL1,EMR}, and local properties
of such solutions have been investigated in \cite{CM,EM,FL2,NV10}, but it is
not known whether they are unique in the class of solutions with finite mass.
It is also not clear, not even on the formal level, whether self-similar
solutions with algebraic decay exist.

In this article we will consider mass-conserving self-similar solutions to
coagulation equations with the following product kernel of homogeneity
$2\lambda\in(0,1)$:
\begin{equation}
K\left(  \xi,\eta\right)  =\left(  \xi\eta\right)  ^{\lambda}%
\  ,\\ \ \ 0<\lambda<\frac{1}{2}\,. \label{S1E2}%
\end{equation}
Existence of  a solution for this kernel
has been established in \cite{FL1}. The purpose of this 
paper is to give an asymptotic description of such
a solution
in the regime
$\lambda\to0$.

Let us describe briefly what is known for self-similar solutions for kernels
as in (\ref{S1E2}) and compare it to results for the so-called sum
kernels. To fix ideas, we restrict ourselves to kernels of the form
\begin{equation}
\label{S1E2b}K(\xi,\eta) = \xi^{\alpha}\eta^{\beta}+\xi^{\beta}\eta^{\alpha},
\qquad\alpha+\beta= 2 \lambda, \; 0\leq\alpha\leq\beta,
\end{equation}
even though most of the results that we mention also apply to more general
kernels with the same growth behavior as the ones in (\ref{S1E2b}). We will
need in particular to distinguish between the case $\alpha=0$, the sum
kernel, and the case $\alpha>0$, the product kernel.
 It has long been predicted \cite{Le1,DE1}
that self-similar solutions for kernels as in (\ref{S1E2b}) exhibit a singular
power-law behavior of the form $x^{-\tau}$ with $\tau<1+2\lambda$ in the case
$\alpha=0$, and $\tau = 1+2\lambda$ for the case $\alpha>0$. This has been
rigorously proved for the case $\alpha=0$ in \cite{FL2} and for the case
$\alpha>0$ in \cite{EM,NV10}. As has been pointed out in \cite{FL2}, the next
order behavior for small clusters in the case $\alpha=0$ can then easily be
established and is as predicted by the physicists. In the case $\alpha>0$,
however, the next order behavior has not been known \cite{DE1}, and only
numerical simulations suggested that it is oscillatory \cite{FilL1,Lee}.

Our goal in this paper is to describe formally how to construct
mass-\linebreak conserving self-similar solutions for  kernels as in
(\ref{S1E2}) in the limit $\lambda\to0$ and by this also describe their
asymptotic behavior on the whole positive real line. While we believe that
such self-similar solutions are unique, our approach only gives the
construction of one such solution. In forthcoming work \cite{NV11} we will
also show how to make this construction rigorous. We will see that in the
limit $\lambda\to0$ the oscillatory character of the solutions becomes very
explicit and can be interpreted in terms of simple ODEs coupled with explicitly
solvable equations.

In the limit $\lambda\rightarrow0^{+}$ the product kernel becomes close
to the sum kernel $x^{2\lambda}+y^{2\lambda}$.
 As described above the power law
for the sum kernel is different from the one for the product
kernel. This is due to the fact that, thinking in terms of coagulating
particles, the physical behaviour of these particles is different in these two
cases. Indeed, the power law obtained for the sum case shows that
particles with small $x$ coalesce mostly with particles that are much bigger
than themselves. On the contrary, in the case of the product kernel
small particles interact mostly with the ones having a comparable size. The
point $\lambda=0$ can be thought as a bifurcation point where both kind of
behaviours take place simultaneously. 
%The results of this paper show that, at
%least for the solutions that we describe, the particles for the kernel
%(\ref{S1E2}) interact with comparable particles but the range of particles
%involved diverges as $\lambda\rightarrow0^{+},$ as it could be expected at
%this bifurcation point.

\section{Preliminaries and Overview}

\subsection{Equation for self-similar solutions}

We are now going to derive the equations that are solved by self-similar
solutions. It is known that the coagulation equation (\ref{S1E1}) can be
written in the following conservative form that makes the
conservation of the number of monomers $\int_{0}^{\infty}\xi f\left(  \xi,t\right)
d\xi$ transparent.
\begin{equation}
\partial_{t}\left(  \xi f\left(  \xi,t\right)  \right)  +\partial_{\xi}\left(
\int_{0}^{\xi}d\eta\int_{\xi-\eta}^{\infty}d\rho K\left(  \eta,\rho\right)  \eta f\left(
\eta,t\right)  f\left(  \rho,t\right)  \right)  =0\,. \label{S2E1}%
\end{equation}

Using the self-similar variables
\begin{equation}
f\left(  \xi,t\right)  =t^{-\frac{2}{1-2\lambda}}g\left(  x\right)
\  ,\ \ \ x=\frac{\xi}{t^{\frac{1}{1-2\lambda}}}\,, \label{S1E3}%
\end{equation}
as well as the form of the kernel (\ref{S1E2}), equation (\ref{S2E1}) becomes
\[
-\partial_{x}\left(  \frac{x^{2}g}{\left(  1{-}2\lambda\right)  }\right)
+\partial_{x}\left(  \int_{0}^{x}dy\int_{x-y}^{\infty}dzK\left(  y,z\right)
yg\left(  y\right)  g\left(  z\right)  \right)  =0\,.
\]
Integrating this equation, assuming decay of the solutions as $x\rightarrow
\infty$ and imposing absence of particle fluxes, we obtain
\begin{equation}
\frac{x^{2}g\left(  x\right)  }{\left(  1{-}2\lambda\right)  }=\int_{0}%
^{x}dy\int_{x-y}^{\infty}dzK\left(  y,z\right)  yg\left(  y\right)  g\left(
z\right)\,.  \label{S2E2}%
\end{equation}
The balance of the terms on both sides of (\ref{S2E2}) suggests the following
power law behaviour for $g$ near the origin,
\begin{equation}
g\left(  x\right)  \sim\frac{H_{\lambda}}{\left(  1{-}2\lambda\right)  }\frac
{1}{x^{1+2\lambda}}\ \ \text{as\ \ }x\rightarrow0^{+}\  ,\ \ \ H_{\lambda
}=\frac{\lambda}{B\left(  1{-}\lambda,1{-}\lambda\right)  } \label{S2E3}%
\end{equation}
where $B\left(  \cdot,\cdot\right)  $ is the classical Beta function
\cite{AS}. For further reference we notice that
\begin{equation}
\label{S2E3b}B\left(  1{-}\lambda,1{-}\lambda\right) \sim1+2\lambda+ o(\lambda) \qquad
\mbox{ as } \lambda \to 0.
\end{equation}

In order to remove the power law behaviour we introduce the new function
\begin{equation}
h\left(  x\right)  =\frac{\left(  1{-}2\lambda\right)  }{H_{\lambda}%
}x^{1+2\lambda}g\left(  x\right)\,.  \label{S2E4}%
\end{equation}
Then $h$ solves
\begin{equation}
h\left(  x\right)  =H_{\lambda}x^{2\lambda-1}\int_{0}^{x}\,dy\,y^{-2\lambda}h\left(
y\right)  \int_{x-y}^{\infty}dzK\left(  y,z\right)  z^{-\left(
1+2\lambda\right)  }h\left(  z\right) \,. \label{S2E5}%
\end{equation}
In the rest of the paper we will study solutions of (\ref{S2E5}) in the limit
$\lambda\to0$ that satisfy $h(x) \to0$ as $x \to\infty$. Notice that
(\ref{S2E5}) has the explicit constant solution $h\left(  x\right)  =1$ that
corresponds to the well-known solution with infinite mass
$g\left(  x\right)  =\frac{H_{\lambda}}{\left(
1-2\lambda\right)  }\frac{1}{x^{1+2\lambda}}$ in the formulation (\ref{S2E2}).

\subsection{Reformulation as a Volterra integro-differential equation}

\label{ShootingProblem}

A main idea in our approach is to reformulate the problem of finding solutions
of (\ref{S2E5}) as a shooting problem for a Volterra integro-differential
equation. Indeed, we  can rewrite (\ref{S2E5}) as
\begin{align}
h\left(  x\right)   &  =\frac{H_{\lambda}}{\lambda\left(  1{-}\lambda\right)
}u\left(  x\right)  v\left(  x\right)  +H_{\lambda}x^{2\lambda-1}\int_{0}%
^{x}dy\int_{x-y}^{x}dzy^{-\lambda}z^{-\left(  \lambda+1\right)  }h\left(
y\right)  h\left(  z\right)\,, \label{S3E4}\\
u\left(  x\right)   &  =\left(  1{-}\lambda\right)  x^{\lambda-1}\int_{0}%
^{x} dy y^{-\lambda}h\left(  y\right)  \  ,\ \ \ \ v\left(  x\right)  =\lambda
x^{\lambda}\int_{x}^{\infty} dz z^{-\left(  \lambda+1\right)  }h\left(  z\right)
\,. \label{S3E5}%
\end{align}
Differentiating (\ref{S3E5}) we obtain the equations
\begin{equation}
xu_{x}=-\left(  1-\lambda\right)  u+\left(  1-\lambda\right)
h\  ,\ \ \ \ xv_{x}=\lambda v-\lambda h\,. \label{S3E7}%
\end{equation}
Thus the solution of (\ref{S2E5}) solves (\ref{S3E7}) with $h$ as in
(\ref{S3E4}), that has the advantage of being a Volterra integro-differential
equation. However, the problem (\ref{S3E4}), (\ref{S3E7}) admits solutions
that in general do not decay as $x\rightarrow\infty$ and therefore they do not
provide solutions of (\ref{S2E5}). 
 Therefore, among all the solutions of (\ref{S3E4}),
(\ref{S3E7}) we must select those satisfying
\begin{equation}
v\left(  x\right)  \rightarrow0\  ,\ \ \ \ h\left(  x\right)  \rightarrow
0\  ,\ \ \ \ u\left(  x\right)  \rightarrow0\ \ \text{ as \ \ }x\rightarrow
\infty\,.\label{S3E7a}%
\end{equation}

\subsection{Overview of different regimes}

We now sketch the asymptotics of the solution that we are going to construct. It
passes, roughly speaking, through three stages that are described in detail in Sections
\ref{S.zero}, \ref{CombRegion} and \ref{S.asymptotics} respectively.

Section \ref{S.zero} discusses the behavior of the solution near $x =0$. As
pointed out before, near the origin the solution is oscillatory and behaves as
\begin{equation}
\label{oscillations}h\left(  x\right)  \sim1+Cx^{\beta\left(  \lambda\right)
}\cos\left(  \alpha\left(  \lambda\right)  \log\left(  x\right)
+\varphi\right)  \ \ \text{as\ \ }x\rightarrow0^{+}
\end{equation}
where $\beta(\lambda) \sim\lambda/2$ and $\alpha(\lambda) \sim\sqrt{\lambda}$
as $\lambda\to0$. Our analysis consists of constructing a solution by starting
with $h$ as in (\ref{oscillations}) and using $K=Ce^{i\varphi}$ as a shooting
parameter. However, since equation (\ref{S2E5}) is invariant under the
rescaling $x \mapsto a x$ for $a>0$, we can restrict the range of $K$ to
$[1,\exp{2 \pi(\beta/\alpha)})$ (see also the comment in Section
\ref{Ss.variables}). In this regime near the origin we can approximate
(\ref{S3E4}) and (\ref{S3E5}) by a nonlinear ODE (see (\ref{S5E6})), that is
a
perturbation of a simple ODE system. For the perturbed system we
can use an adiabatic approximation to compute the increase of an associated
energy $E$ along trajectories. As we discuss in Section \ref{Ss.validity} this
first ODE approximation is valid as long as $E \ll1/\lambda$. When $E
\sim1/\lambda$ we enter a new regime that is described in Section
\ref{CombRegion}. In this regime $h$ develops peaks of height $1/\lambda$ and
width of order one that are separated by wide regions in which $h$ is small.
 The regime where $h$ is small is again described by an ODE
(cf. Section \ref{Ss.ode}), while the peaks are described by an
integro-differential equation (see Section \ref{Ss.integro}). The analysis of
these respective regimes is done in Sections \ref{Ss.intanalysis} and
\ref{Ss.odeanalysis}, while their coupling is described in
Section \ref{asymptPeaks}. In Section \ref{S.shooting} we will then show by a
continuity argument that there exists a shooting parameter such that the
corresponding solution $h$ converges to zero as $x \to\infty$. In the third
regime, that is discussed in Section \ref{S.asymptotics}, 
$h$ then decays exponentially fast to zero. 

\section{The behaviour as $x \to0$}

\label{S.zero}

\subsection{Oscillations}

It has been observed in \cite{FilL1,Lee} that the self-similar solutions
associated with (\ref{S1E1}) exhibit oscillations for small values of $x.$ This
oscillatory behaviour can be seen as follows. If we make the ansatz
$h(x)=1+\alpha x^{\mu} + \cdots $  with some $\mu\in\mathbb{C}$,
plug this into (\ref{S2E5}) and keep only the leading order terms, we obtain
that $\mu$ must satisfy
\[%
\begin{split}
\frac{1}{H_{\lambda}}&  =\int_{0}^{1} dy y^{-2\lambda}\int_{1-y}^{\infty}dzK\left(
y,z\right)  z^{-\left(  1+2\lambda\right)  +\mu}
\\
&
+\int_{0}^{1} dy y^{-2\lambda+\mu}\int_{1-y}^{\infty}dzK\left(
y,z\right)  z^{-\left(  1+2\lambda\right)  }\,.%
\end{split}
\]
After some computations, we find that this is equivalent to 
\begin{equation}
\Psi_{\lambda}\left(  \mu\right)  : =\left(  \lambda{-}\mu\right)  B\left(
1{-}\lambda,1{-}\lambda\right)  -\left(  2\lambda{-}\mu\right)  B\left(
1{-}\lambda,1{-}\lambda{+}\mu\right)  =0 \,. \label{S3E3}%
\end{equation}

One can show that $\Psi_{\lambda}$ has two complex conjugate roots 
$\mu_{\lambda}^{\pm}=\pm\alpha\left(  \lambda\right)  i+\beta\left(
\lambda\right)  $ with positive real part and nonzero imaginary part. Indeed, in the limit
$\lambda \to 0$ this can be seen easily as the dominating terms in (\ref{S3E3}) are,
assuming $|\mu| \gg \lambda$, 
\[
\Psi_{\lambda}\left(  \mu\right) \sim (\lambda {-} \mu) (1{+}2\lambda) - (2\lambda{-}\mu)(1
{-}\mu{+}2\lambda + \mu^2)
\]
and thus the roots of $\Psi_{\lambda}$ satisfy
\begin{equation}
\mu_{\lambda}^{\pm}=\frac{\lambda}{2}\pm\sqrt{\lambda}i+O\left(
\lambda^{\frac{3}{2}}\right)  \ \ \text{as\ \ }\lambda\rightarrow0^{+}.
\label{D1E2}%
\end{equation}

Consequently we expect the following asymptotics for the
solutions of (\ref{S2E5}):%
\begin{equation}
h\left(  x\right)  \sim1+Cx^{\beta\left(  \lambda\right)  }\cos\left(
\alpha\left(  \lambda\right)  \log\left(  x\right)  +\varphi\right)
\ \ \text{as\ \ }x\rightarrow0^{+} \label{S2E6}%
\end{equation}
for suitable real constants $C$ and $\varphi.$ 

Notice that the asymptotics (\ref{S2E6}) indicate that two consecutive local
maxima of $h$, called  $x_{1}$ and $ x_{2}$, satisfy
$
\frac{x_{1}}{x_{2}}\approx\exp\left(  \frac{2\pi}{\sqrt{\lambda}}\right)$
for small $\lambda.$ In comparison, to double the amplitude of (\ref{S2E6}) we
must multiply $x$ by a number of order $\exp\left(  \frac{2\log\left(
2\right)  }{\lambda}\right)  $. Thus, the behaviour (\ref{S2E6}) is
essentially oscillatory, with a growth in the amplitude that takes place on a
much larger scale.

\subsection{A new set of variables}

\label{Ss.variables}

For our forthcoming analysis it will be more convenient to use the following variables:
\begin{align}
x  &  =e^{X}  ,\ \ \ y=e^{Y}  ,\ \ \ z=e^{Z}\,,\label{S3E8}\\
h\left(  x\right)   &  =H\left(  X\right)  ,\ \ \ \ u\left(  x\right)
=U\left(  X\right),  \ \  \ \ v\left(  x\right)  =V\left(  X\right)\,.
\label{S3E9}%
\end{align}
Then (\ref{S3E4}), (\ref{S3E7}) becomes
\begin{align}
H   &  =\frac{H_{\lambda}}{\lambda\left(  1{-}\lambda\right)}
\,U  V  +I\left[  H\right] \,,
\label{S3E10}\\
\frac{dU}{dX} &   =-\left(  1{-}\lambda\right)  U  +\left(
1{-}\lambda\right)  H\,, \qquad
\frac{dV}{dX}    =\lambda V  -\lambda H\,,
\label{S3E12}%
\end{align}
where
\begin{equation}
I\left[  H\right]  \left(  X\right)  =H_{\lambda}\int_{-\infty}^{X}%
dYe^{\left(  1-\lambda\right)  \left(  Y-X\right)  }H\left(  Y\right)
\int_{\log\left(  e^{X}-e^{Y}\right)  }^{X}dZe^{-\lambda\left(  Z-X\right)
}H\left(  Z\right)\,.  \label{S3E13}%
\end{equation}

The system (\ref{S3E10})-(\ref{S3E13}) has two explicit solutions, namely
$(H,U,V)$ \linebreak 
$ \equiv(1,1,1)$ and $(H,U,V) \equiv(0,0,0)$. Our goal is to construct
a solution of (\ref{S3E10})-(\ref{S3E13}) for small values of $\lambda$ with
the property
\begin{align}
\lim_{X\rightarrow-\infty}\left(  H\left(  X\right)  ,U\left(  X\right)
,V\left(  X\right)  \right)   &  =\left(  1,1,1\right)\,, \label{S4Ea}\\
\lim_{X\rightarrow\infty}\left(  H\left(  X\right)  ,U\left(  X\right)
,V\left(  X\right)  \right)   &  =\left(  0,0,0\right) \,. \label{S4Eb}%
\end{align}
We will see that a solution that satisfies (\ref{S4Eb}) decays exponentially fast
to zero. This in particular implies that the corresponding self-similar solution
$g$ has finite mass.

Equations
(\ref{S3E8}), (\ref{S3E9}) and (\ref{S2E6}) yield the following asymptotics of
the solutions as $X\rightarrow-\infty$:
\begin{equation}
\left(  H\left(  X\right)  ,U\left(  X\right)  ,V\left(  X\right)  \right)
\sim\left(  1,1,1\right)  +\operatorname{Re}\left(  K\left(  1,\frac{\left(
1{-}\lambda\right)  }{\left(  1{-}\lambda\right)  +\mu^{+}},\frac{-\lambda
}{\left(  \mu^{+}{-}\lambda\right)  }\right)  e^{\mu^{+}X}\right)\,,
\label{S4E0}%
\end{equation}
where $\mu^{+}$ is as in (\ref{D1E2}) and $K=Ce^{i\varphi}%
\in\mathbb{C}.$

The complex number $K$ determines the solution of (\ref{S3E10}%
)-(\ref{S3E13}) uniquely. However, we notice that (\ref{S3E10})-(\ref{S3E13}) is
invariant under translations $X\rightarrow X+a$ 
 and thus all the complex
numbers in the spiral $\mathcal{S}_{\lambda,K}=\left\{  Ke^{\mu^{+}a}%
:a\in\mathbb{R}\right\}  \subset\mathbb{C}$\ yield the same solution up to
translations. Thus we can identify the solutions of
(\ref{S3E10})-(\ref{S3E13}) satisfying (\ref{S4Ea}), up to translations, with
the set of positive real numbers contained between two consecutive
intersections of $\mathcal{S}_{\lambda}$ with the real axis. In particular,
given $K=1$ we have, since $\mu_{\lambda}^{+}=\alpha\left(  \lambda\right)
i+\beta\left(  \lambda\right)$, that the next consecutive point  in
$\mathbb{R}^{+}\cap\mathcal{S}_{\lambda,1}$ is $\exp\left(  \frac{2\pi
\beta\left(  \lambda\right)  }{\alpha\left(  \lambda\right)  }\right)  .$
Therefore, there is a one-to-one correspondence between the points in the
interval $\left[  1,\exp\left(  \frac{2\pi\beta\left(  \lambda\right)
}{\alpha\left(  \lambda\right)  }\right)  \right)  $ and the solutions of
(\ref{S3E10})-(\ref{S3E13}) satisfying (\ref{S4Ea}).

The main result of this paper is that we  show, using
asymptotic arguments, that for small $\lambda$   a value of 
$K\in\left[  1,\exp\left(  \frac{2\pi\beta\left(
\lambda\right)  }{\alpha\left(  \lambda\right)  }\right)  \right)  $ exists such
that (\ref{S4Eb}) holds.

\begin{figure}[ht!]
\centering{%
\includegraphics[width=.46\textwidth]{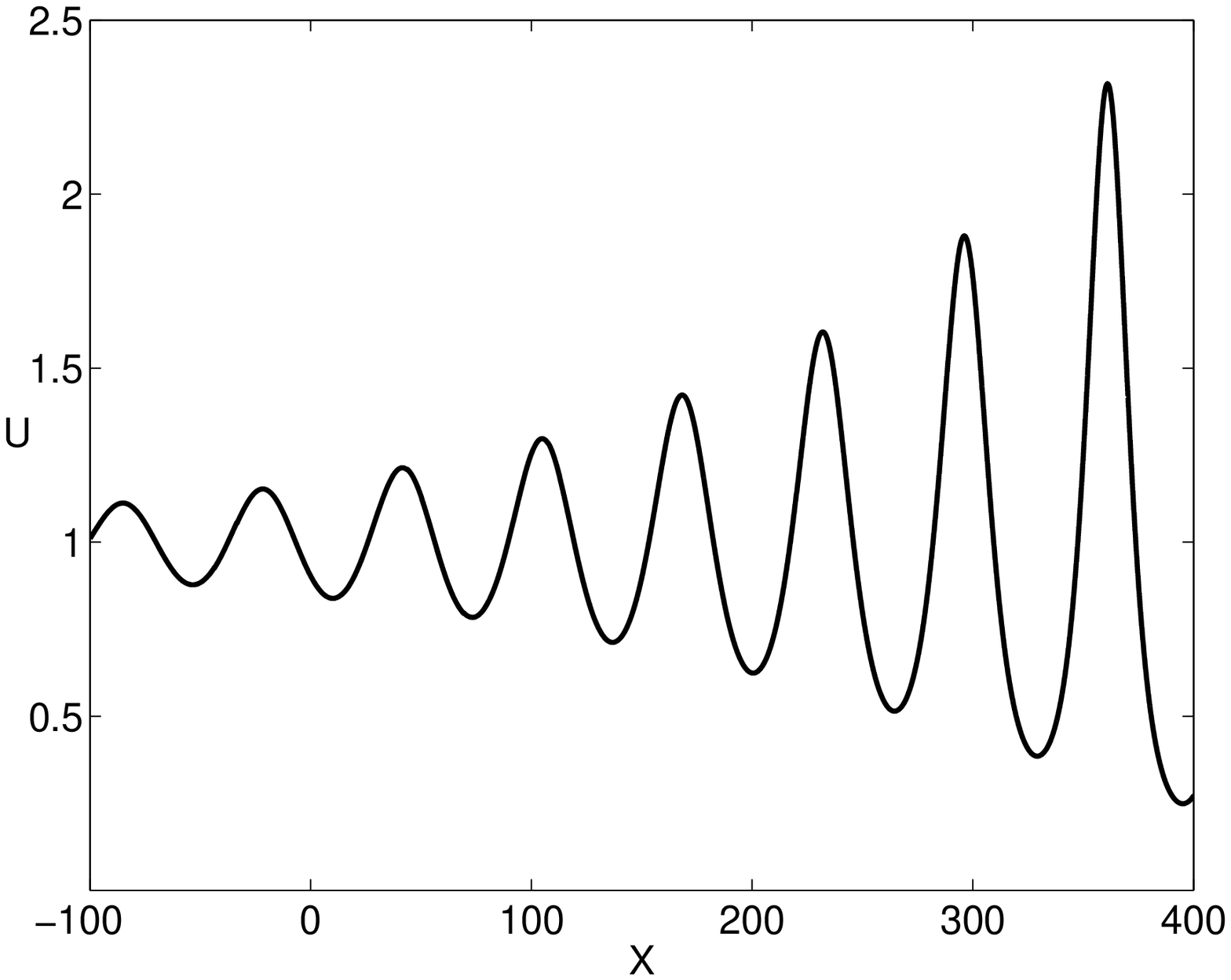}
\includegraphics[width=.46\textwidth]{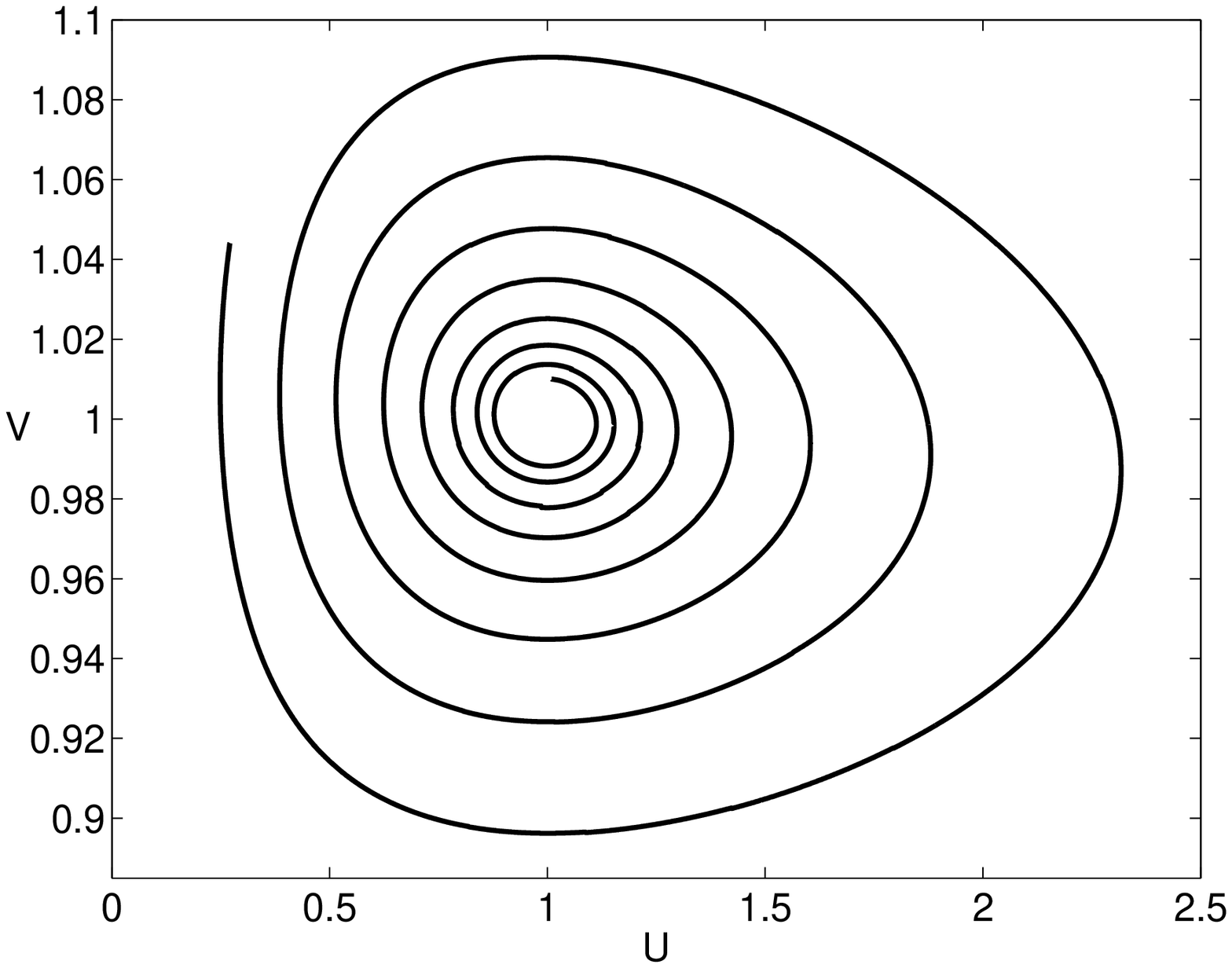}
}%
\caption{Oscillatory behavior of $U$ (left) and phase plane for $U,V$ (right).}
\label{figure1}
\end{figure}

\subsection{The ODE regime\label{ODELolkaVolterra}}

The key idea in computing the asymptotics of the solutions of (\ref{S3E10}%
)-(\ref{S3E13}) is to obtain suitable approximations of the operator $I\left[
H\right]  \left(  X\right)  $ in (\ref{S3E13}) for small $\lambda$. The
formula for $H_{\lambda}$ (cf. (\ref{S2E3})), combined with (\ref{S2E3b}) and
(\ref{S3E13}), suggests the approximation
\begin{equation}
I\left[  H\right]  \left(  X\right)  =\lambda\int_{-\infty}^{X}e^{\left(
Y-X\right)  }H\left(  Y\right)  dY\int_{\log\left(  e^{X}-e^{Y}\right)  }%
^{X}H\left(  Z\right)  dZ\,. \label{E1}%
\end{equation}
We will discuss the consistency of this assumption in Section \ref{S.consistency}.

The operator $I\left[  H\right]  \left(  X\right)  $ can be further approximated 
if $Y/X$ changes significantly faster than $H(Y)/H(X)$, as we saw is the case for
$X \to -\infty$. Then it is natural to approximate $H(Y)$ by $H(X)$ and we obtain
%%%%%%%%%%%%%%%%%%%
\ignore{
If the characteristic length in which
$\left(  H\left(  Y\right)  {-} 1\right)  $ has relevant variations compared
with itself for $Y$ of order $X$ is much larger than one, it would be natural
to approximate $H\left(  Y\right)  $ by $H\left(  X\right)  .$
}
%%%%%%%%%%%%%%%%%%%%%%%%%%%%%%%%%%%%%%%%%%%55
\[
I\left[  H\right]  \left(  X\right)  \sim\lambda H\left(  X\right)
\int_{-\infty}^{X}e^{\left(  Y-X\right)  }dY\int_{\log\left(  e^{X}%
-e^{Y}\right)  }^{X}H\left(  Z\right)  dZ\ \
\]
and changing the order of the integrals
\begin{equation}
I\left[  H\right]  \left(  X\right)  \sim\lambda H\left(  X\right)
\int_{-\infty}^{X}H\left(  Z\right)  e^{Z-X}dZ\,. \label{S5E1}%
\end{equation}

From (\ref{S3E12}) we have
\[
U\left(  X\right)  =\left(  1{-}\lambda\right)  \int_{-\infty}^{X}H\left(
Z\right)  e^{\left(  1{-}\lambda\right)  \left(  Z-X\right)  }dZ \sim
\int_{-\infty}^{X}H\left(  Z\right)  e^{\left(  Z-X\right)  }dZ
\]
which together with (\ref{S5E1}) leads to
\begin{equation}
I\left[  H\right]    \sim\lambda H  U
 \,.  \label{S5E2}%
\end{equation}
Using (\ref{S5E2}) as well as (\ref{S2E3}) and (\ref{S2E3b})  we obtain the
following approximation of (\ref{S3E10}) up to order $\lambda$:
\begin{equation}
H =U V +\lambda U V \left(  U -1\right)  \label{S5E3a}%
\end{equation}
and plugging this approximation into (\ref{S3E12}), (\ref{S3E13}) we obtain%
\begin{align}
\frac{dU}{dX}  &  =-\left(  1{-}\lambda\right)  U + \left( 1{-}\lambda\right)
U  V +\lambda\left(  1{-}\lambda\right)  U V \left(  U -1\right)\,,
\label{S5E3}\\
\frac{dV}{dX}  &  =\lambda V -\lambda U V -\lambda^{2}U V \left(  U -1\right)\,.
\label{S5E4}%
\end{align}
As (\ref{D1E2}) and (\ref{S4E0})
 imply that the changes of $\left(  V{-}1\right)  $ compared to those of
$\left(  U{-}1\right)  $ are of order $\sqrt{\lambda}$, we make  the
 change of variables
\begin{equation}
V=1+\sqrt{\lambda}\omega\ \ ,\ \ \xi=\sqrt{\lambda}X\,, \label{S5E5}%
\end{equation}
that transform (\ref{S5E3}), (\ref{S5E4}) up to order $\sqrt{\lambda}$ into
\begin{equation}
\frac{dU}{d\xi}=U\omega+\sqrt{\lambda}U\left(  U-1\right)  \   ,\ \ \ \ \frac
{d\omega}{d\xi}=1-U+\sqrt{\lambda}\omega\left(  1-U\right) \,. \label{S5E6}%
\end{equation}

\subsection{Analysis of the ODE (\ref{S5E6})\label{ODE}}

Let us denote by $\left(  U_{\lambda},\omega_{\lambda}\right)  $ any solution
of (\ref{S5E6}). We first notice that the asymptotics of the solutions of
(\ref{S5E6}) as $\left(  U_{\lambda},\omega_{\lambda}\right)  \rightarrow
\left(  1,0\right)  $ agree with those obtained in (\ref{S4E0}). Indeed, in
the limit $\lambda\rightarrow0,$ we can compute the asymptotics of $\left(
U,\omega\right)  $ as $\xi\rightarrow-\infty$ using (\ref{S4E0}), (\ref{D1E2})
and (\ref{S5E5}):
\begin{equation}
\left(  (U_{\lambda}{-}1)\left(  \xi\right)  ,\omega_{\lambda}\left(  \xi\right)
\right)    \sim\operatorname{Re}\left(  K\left(
1,i\right)  e^{\left(  i+O\left(
\sqrt{\lambda}\right)  \right)  \xi}\right)  \ \ \text{as}\ \ \xi\rightarrow
-\infty\,.\label{S6E2}%
\end{equation}
Linearizing (\ref{S5E6}) near $\left(  U,\omega\right)  =\left(  1,0\right)  $
we obtain the same asymptotics for the functions $\left(  U_{\lambda}\left(
\xi\right)  ,\omega_{\lambda}\left(  \xi\right)  \right)  $ as $\xi
\rightarrow-\infty.$ This gives the desired matching between the solutions of
the linearization of (\ref{S3E10})-(\ref{S3E13}) around $\left(  H,U,V\right)
=\left(  1,1,1\right)  $ and the solutions of the approximated problem
(\ref{S5E6}).

The approximation (\ref{S6E2}) is valid as long as $\left\vert \left(
U_{\lambda}-1\left(  \xi\right)  ,\omega_{\lambda}\left(  \xi\right)  \right)
\right\vert $ is small. However, a detailed analysis of the nonlinear problem
(\ref{S5E6}) is needed if \linebreak$\left\vert \left(  U_{\lambda}-1\left(
\xi\right)  ,\omega_{\lambda}\left(  \xi\right)  \right)  \right\vert $
becomes of order one. In order to describe the solutions of (\ref{S5E6}) in
this regime, we notice that this equation is a perturbation of the 
 equation
\begin{equation}
\frac{dU_{0}}{d\xi}=U_{0}\omega_{0}\ ,\ \ \ \ \frac{d\omega_{0}}{d\xi
}=1-U_{0} \label{S6E3}\,.
\end{equation}
This equation can be explicitly integrated, since the following quantity is
conserved along trajectories:
\begin{equation}
E=-\log\left(  U_{0}\right)  +\left(  U_{0}-1\right)  +\frac{\omega_{0}^{2}%
}{2} \,. \label{S6E4}%
\end{equation}

The solutions of (\ref{S6E3}) are periodic, but this behavior is not
compatible with the exponential growth obtained in (\ref{S6E2}). This growth
is due to the increase in $E$ produced by the terms of order $\sqrt{\lambda} $
in (\ref{S5E6}). We now compute this change of energy for values of $\left(
U,\omega\right)  $ of order one. Using (\ref{S5E6}) and (\ref{S6E4}) we
obtain
\begin{equation}
\frac{dE}{d\xi}=\sqrt{\lambda}\left[  \left(  U_{\lambda}-1\right)
^{2}+\omega_{\lambda}^{2}\left(  1-U_{\lambda}\right)  \right]  \label{S6E5}\,.
\end{equation}

Since the solutions of (\ref{S5E6}) are close to those of (\ref{S6E3}) during
finite time intervals, we can  adiabatically compute the change in $E.$ More
precisely, the solutions of (\ref{S6E3}) satisfying (\ref{S6E4}) are periodic
with a period $T\left(  E\right)  $ given by
\begin{equation}
T\left(  E\right)  =\sqrt{2}\int_{U_{-}\left(  E,0\right)  }^{U_{+}\left(
E,0\right)  }\frac{d\eta}{\eta\sqrt{E+\log\left(  \eta\right)  -\left(
\eta-1\right)  }} \label{S6E5a}%
\end{equation}
where the functions $U_{-}\left(  E,\omega\right)  \leq U_{+}\left(
E,\omega\right)  $ are defined as the roots of the equation
\begin{equation}
\log\left(  U\right)  -\left(  U-1\right)  =\frac{\omega^{2}}{2}-E
\label{S6E6}%
\end{equation}
for any given value of $\omega$ and $E\geq\frac{\omega^{2}}{2}.$
Integrating (\ref{S6E5}) and using the adiabatic approximation, we can then
approximate the change of $E$ during an interval of length $T\left(  E\right)
$ as
\begin{equation}
\frac{E\left(  T\left(  E\right)  \right)  -E\left(  0\right)  }{\sqrt
{\lambda}}=\int_{0}^{T\left(  E\right)  }\left[  \left(  U_{0}\left(
\xi\right)  -1\right)  ^{2}+\omega_{0}^{2}\left(  \xi\right)  \left(
1-U_{0}\left(  \xi\right)  \right)  \right]  d\xi=\Phi\left(  E\right)
\label{S6E7}%
\end{equation}
where $\left(  U_{0}\left(  \xi\right)  ,\omega_{0}\left(  \xi\right)
\right)  $ is a solution of (\ref{S6E3}) satisfying (\ref{S6E4}).

It turns out that the function $\Phi\left(  E\right)  $ is positive for any
$E>0.$ This follows, using
the second
equation in (\ref{S6E3}), via
\[
\Phi\left(  E\right)  =\int_{-\sqrt{2E}}^{\sqrt{2E}}\left[  \left(
1-U_{-}\left(  E,\omega\right)  \right)  +\omega^{2}\right]  d\omega
-\int_{-\sqrt{2E}}^{\sqrt{2E}}\left[  \left(  1-U_{+}\left(  E,\omega\right)
\right)  +\omega^{2}\right]  d\omega\,,
\]
whence
\begin{equation}
\Phi\left(  E\right)  =2\int_{0}^{\sqrt{2E}}\left(  U_{+}\left(
E,\omega\right)  -U_{-}\left(  E,\omega\right)  \right)  d\omega>0\,.
\label{S6E8}%
\end{equation}

We can compute the asymptotics of (\ref{S6E8}) as $E\rightarrow\infty$ using
(\ref{S6E6}). Notice that the leading contribution to the integral is given by
$U_{+}\left(  E,\omega\right)  ,$ since $U_{-}\left(  E,\omega\right)  \leq1.$
Then
\begin{equation}
\Phi\left(  E\right)  \sim2\int_{0}^{\sqrt{2E}}\left(  E-\frac{\omega^{2}}%
{2}\right)  d\omega=\frac{4\sqrt{2}}{3}E^{\frac{3}{2}}\ \ \text{as\ \ }%
E\rightarrow\infty\,.\label{S6E8a}%
\end{equation}

\subsection{Range of validity of the ODE regime}

\label{Ss.validity}

 We have obtained that, as long as the approximation
(\ref{S5E6}) is valid, we can approximate the functions $\left(  U_{\lambda
}\left(  \xi\right)  ,\omega_{\lambda}\left(  \xi\right)  \right)  $ by 
 $\left(  U_{0}\left(  \xi\right)  ,\omega_{0}\left(  \xi\right)  \right)  $
while  $E$ is computed via the iterative recursion (\ref{S6E7}). 
Due to (\ref{S6E8a}) these values of $E$  increase an amount
of order $\sqrt{\lambda}\left(  1+E\right)  ^{\frac{3}{2}}$ in each period of
size $T\left(  E\right)  .$ We now proceed to discuss the range of validity of
the different approximations that have been used to derive (\ref{S6E7}) and
 (\ref{S5E6}).

To examine the range
 of values of $E$ for which the adiabatic
approximation is valid, we need to show that the value of $E\left(
\xi\right)  $ remains close to $E\left(  0\right)  $ for $0\leq\xi\leq
T\left(  E\right)  .$ Due to (\ref{S6E6}) and (\ref{S6E8}) the
adiabatic approximation is valid as long as
\begin{equation}
\sup_{0\leq\xi\leq T\left(  E\right)  }\left\vert E\left(  \xi\right)
-E\left(  0\right)  \right\vert \ll E\left(  0\right)\,.  \label{S6E9}%
\end{equation}
Let us remark that (\ref{S6E9}) does not follow immediately from the
inequality $\left\vert E\left(  T\left(  E\left(  0\right)  \right)  \right)
-E\left(  0\right)  \right\vert \ll E\left(  0\right)  $ because the right-hand
side of (\ref{S6E5}) does not have a sign and therefore it is not clear that
\begin{equation}
\sup_{0\leq\xi\leq T\left(  E\right)  }\left\vert E\left(  \xi\right)
-E\left(  0\right)  \right\vert \leq C\left\vert E\left(  T\left(  E\left(
0\right)  \right)  \right)  -E\left(  0\right)  \right\vert \label{S6E10}%
\end{equation}
for some $C>0.$ A priori it is not possible to rule out the possibility of big
changes of $E\left(  \xi\right)  $ along each cycle balanced in such a way
that the overall change of $E$ along the cycle is just $\Phi\left(  E\right)
.$ However, as shown in Appendix \ref{A.energy}, it turns out that
(\ref{S6E10}) holds true.

Using (\ref{S6E7}), (\ref{S6E8a}) and (\ref{S6E10}) it  follows that the
adiabatic approximation is valid as long as $\sqrt{\lambda}\left(  1+E\right)
^{\frac{3}{2}}\ll \left(  1+E\right) $ and this is satisfied as long as
\begin{equation}
\lambda\left(  1+E\right)  \ll 1 \label{S6E11}\,.
\end{equation}

We now study the range of validity of the approximation (\ref{S5E2}) that is
the main ingredient in deriving the approximate problem (\ref{S5E6}). The main
assumption used in the derivation of (\ref{S5E2}) is that the characteristic
length scale for which $\left(  H\left(  X\right)  -1\right)  $ has changes
comparable to itself is much larger than one.
Due to (\ref{S5E3a}) it follows that as long as $V$ remains close to $1$ the
 characteristic length scale for
$\left(  H-1\right)  $ is the same as for $\left(  U-1\right)  .$ Notice that
$\left(  V-1\right)  $ is small if $\sqrt{\lambda}\omega\ll 1.$ 
Since  
we have $\frac 1 2 \omega^2 \leq E$, the condition $\sqrt{\lambda} \omega \ll 1$
holds if (\ref{S6E11}) is true, and so
 the characteristic
length scales for $\left(  H-1\right)  $ and $\left(  U-1\right)  $ are the
same under the assumption (\ref{S6E11}). For $\left\vert \left(
U,\omega\right)  \right\vert $ of order one, the evolution of (\ref{S6E3})
takes place in the length scale $\xi,$ or equivalently for changes in $X$ of
order $\frac{1}{\sqrt{\lambda}}\gg 1.$ Therefore, the condition in (\ref{S6E11})
that allows us to
obtain (\ref{S5E2}) is immediately satisfied and we can restrict our analysis
to the case $\left\vert \left(  U,\omega\right)  \right\vert \gg 1. $ The first
equation in (\ref{S6E3}) combined with (\ref{S5E5}) implies that
\begin{equation}
\left\vert \frac{d\left(  \log\left(  U\right)  \right)  }{dX}\right\vert
=\sqrt{\lambda}\left\vert \omega\right\vert \label{S6E12}\,.
\end{equation}

The left-hand side measures the relative variations in $U$ with respect to
changes in $X.$ The characteristic length scale that describes the changes in
this quantity is large as long as $\sqrt{\lambda}\left\vert \omega\right\vert
\ll 1$ which is again true if  (\ref{S6E11}) holds.

Therefore, if the condition (\ref{S6E11}) is satisfied, we can apply
simultaneously all the approximations that lead to (\ref{S5E6}), the adiabatic
approximation yielding (\ref{S6E7}) as well as the condition $\left\vert
V-1\right\vert \ll 1.$ However, the three assumptions break down simultaneously
if $E$\ becomes of order $\frac{1}{\lambda}$. Since (\ref{S6E7}) and
(\ref{S6E8a}) imply that $E$ increases to arbitrarily large values, the
failure of (\ref{S6E11}) happens for every solution of (\ref{S3E10}%
)-(\ref{S3E13}) satisfying (\ref{S4E0}).

\section{ The intermediate regime}
\label{CombRegion}

In order to describe the solutions of (\ref{S3E10})-(\ref{S3E13}) when
$\lambda E\approx1$  we introduce a new group of variables. It turns out to be
convenient to split the regime into  the one  where $U$ (or
$H$) is of order one or smaller, and the one where $U$ (or $H$) is large.
Notice that (\ref{S6E6}) implies that for a given value of $E\sim\frac
{1}{\lambda}$ the  minimum of $U$ scales as $e^{-1/\lambda}$, and the
maximum scales as $1/\lambda$.  We remark that the
forthcoming analysis will show that as long as $V$ remains of order one and
is not small the
values of $H$ and $U$ have the same order of magnitude. This assumption will
be made implicitly in all the remaining computations and will be justified by the
self-consistency of the derived asymptotics.

\subsection{The integro-differential equation regime}

\label{Ss.integro}

We begin by studying the case in which $U$ (and $H$) are large. Since we are
interested in the case $E\approx\frac{1}{\lambda}$, equation (\ref{S6E6}) suggests
the rescaling $U\approx\frac{1}{\lambda},\ \omega\approx\frac{1}{\sqrt
{\lambda}}.$ Then (\ref{S6E12}) yields a characteristic length scale for $X$
of order one. This suggests  introducing the  new set of variables
\begin{equation}
U\left(  X\right)  =\frac{1}{\lambda}\mathcal{U}\left(  X\right)
  , \ \ \ V\left(  X\right)  =\mathcal{V}\left(  X\right)    ,\ \ \ H\left(
X\right)  =\frac{1}{\lambda}\mathcal{H}\left(  X\right) \,.
\ \label{S7E1}%
\end{equation}
Plugging
(\ref{S7E1}) into (\ref{S3E10})-(\ref{S3E13}) we obtain
\begin{align}
\mathcal{H}  &  =\frac{H_{\lambda}}{\lambda\left(
1{-}\lambda\right)  }\mathcal{U}  \mathcal{V}
+\frac{1}{\lambda}I\left[  \mathcal{H}\right]\,,  \label{S7E4}\\
\frac{d\mathcal{U}}{dX}  &  =-\left(  1{-}\lambda\right)  \mathcal{U}
  +\left(  1{-}\lambda\right)  \mathcal{H}\,,
\, \qquad \frac{d\mathcal{V}}{dX}    =\lambda\mathcal{V}
-\mathcal{H}\,,  \label{S7E6}%
\end{align}
where the operator $I\left[  \mathcal{H}\right]  $ is as in (\ref{S3E13}).
Using the approximation (\ref{E1}) that will be seen to be still valid in this
region we obtain
\[
\frac{1}{\lambda}I\left[  \mathcal{H}\right]  \left(  X\right)  =\left(
1+O\left(  \lambda\right)  \right)  \int_{-\infty}^{X}e^{\left(  Y-X\right)
}\mathcal{H}\left(  Y\right)  dY\int_{\log\left(  e^{X}-e^{Y}\right)  }%
^{X}\mathcal{H}\left(  Z\right)  dZ\,.
\]

Taking the limit $\lambda\rightarrow0$ in (\ref{S7E4})-(\ref{S7E6}) we obtain
\begin{align}
\mathcal{H}\left(  X\right)   &  =\mathcal{U}\left(  X\right)  \mathcal{V}%
\left(  X\right)  +\int_{-\infty}^{X}e^{\left(  Y-X\right)  }\mathcal{H}%
\left(  Y\right)  dY\int_{\log\left(  e^{X}-e^{Y}\right)  }^{X}\mathcal{H}%
\left(  Z\right)\,,  dZ\label{S7E7}\\
\frac{d\mathcal{U}}{dX}  &  =-\mathcal{U}  +\mathcal{H}
\  ,\ \ \ \frac{d\mathcal{V}}{dX}=-\mathcal{H}\,.
\label{S7E9}%
\end{align}

\subsection{The ODE regime}
\label{Ss.ode}

If $\lambda E\approx1$ but $U$ (and $H$)  remains bounded we can
approximate (\ref{S3E10})-(\ref{S3E13}) as follows. 
Keeping just the leading terms in
(\ref{S3E10}), (\ref{S3E13}) we arrive at the approximation $H=UV$.
Plugging this approximation
into (\ref{S3E12}) we obtain
\begin{equation}
\frac{dU}{dX}=-U  +U  V
\ \ ,\ \ \ \ \frac{dV}{dX}=\lambda\left(  V  -U
 V \right) \,. \label{S7E10}%
\end{equation}

\subsection{Explicit solution of (\ref{S7E7}), (\ref{S7E9})}
\label{Ss.intanalysis}

We are going to compute the solutions of (\ref{S7E7}), (\ref{S7E9})  explicitly using
Laplace transforms. In order to bring (\ref{S7E7}), (\ref{S7E9}) into the form
of a convolution equation we introduce the set of variables
\begin{equation}
x=e^{X}\  ,\ \ \mathcal{H}\left(  X\right)  =\bar{h}\left(  x\right)
\  ,\ \ \ \mathcal{U}\left(  X\right)  =\bar{u}\left(  x\right)
\  ,\ \ \mathcal{V}\left(  X\right)  =\bar{v}\left(  x\right)  \label{S8E0}%
\end{equation}
that transforms (\ref{S7E7}), (\ref{S7E9}) into
\begin{align}
\bar{h}\left(  x\right)   &  =\bar{u}\left(  x\right)  \bar{v}\left(
x\right)  +\frac{1}{x}\int_{0}^{x}\bar{h}\left(  y\right)  dy\int_{x-y}%
^{x}\frac{\bar{h}\left(  z\right)  }{z}dz\,,\label{S8E1}\\
x\bar{u}_{x}  &  =-\bar{u}+\bar{h}\ \ \ ,\ \ \ x\bar{v}_{x}=-\bar
{h}\,.
\label{S8E2}%
\end{align}

Notice that this system of equations can be obtained from (\ref{S3E4}),
(\ref{S3E7}) by means of the rescaling $h=\frac{\bar{h}}{\lambda}%
,\ u=\frac{\bar{u}}{\lambda},\ v=\bar{v}$ and taking the limit $\lambda
\rightarrow0^{+}.$ We can compute explicitly a family of solutions of
(\ref{S8E1}), (\ref{S8E2}) that will be used to describe the solutions of
(\ref{S3E10})-(\ref{S3E13}) in the limit $\lambda\rightarrow0^{+}.$

\begin{theorem}
\label{solInt}

For any given $a \in (0,2)$ and $\kappa>0$ there exists a solution
to (\ref{S8E1}), (\ref{S8E2}) that satifies
\begin{equation}
\bar h(0)=\bar h(\infty)= \bar u(0)= \bar u(\infty)=0\,, \ \
\lim_{x\rightarrow0^{+}}\bar{v}\left(  x\right)  =1+a  , \ \ \lim
_{x\rightarrow\infty}\bar{v}\left(  x\right)  =1-a\,. \label{F7E7}%
\end{equation}
It is given by
\begin{equation}
\bar{v}\left(  x\right)  =\frac{1}{2\pi i}\int_{\gamma}\left[  1+a\left(
\frac{\zeta^{a}-\kappa x^{a}}{\zeta^{a}+\kappa x^{a}}\right)  \right]
\frac{e^{\zeta}}{\zeta}d\zeta\label{F7E5}%
\end{equation}
where $\gamma=\left\{  \gamma\left(  s\right)  :s\in\mathbb{R}\right\}  $ is a
contour in the complex plane contained in the domain $\left\{  \arg\left(
\zeta\right)  \in\left(  -\theta_{0},\theta_{0}\right)  \right\}  $ satisfying
$\arg\left(  \gamma\left(  s\right)  \right)  \rightarrow-\theta_{0}$ as
$s\rightarrow-\infty$ and $\arg\left(  \gamma\left(  s\right)  \right)
\rightarrow\theta_{0}$ as $s\rightarrow\infty$ with $\frac{\pi}{2}<\theta
_{0}<\min\left\{  \frac{\pi}{a},\pi\right\}  .$ 

If $0<a<1$ we have $\bar{v}_{x}\left(  x\right)  <0,\ \bar{h}\left(  x\right)
>0 $ and $\ \bar{u}\left(  x\right)  >0$ for any $x\in\mathbb{R}^{+}.$

The following asymptotics hold:
\begin{align}
\bar{v}_{x}\left(  x\right)   &  \sim-\frac{2a\kappa}{\Gamma\left(  a\right)
}\frac{1}{x^{1-a}}\ \ \ \ \text{as\ \ }x\rightarrow0^{+}\ \ ,\ \ a\in\left(
0,2\right) \label{G1E1}\\
\bar{v}_{x}\left(  x\right)   &  \sim\frac{2a}{\kappa\Gamma\left(  -a\right)
}\frac{1}{x^{a+1}}\ \ \ \ \text{as\ \ }x\rightarrow\infty\ \ \ ,\ \ a\in
\left(  0,1\right)  \cup\left(  1,2\right) \nonumber
\end{align}%
\begin{align}
\bar{h}\left(  x\right)   &  \sim\frac{2a\kappa}{\Gamma\left(  a\right)
}x^{a}\ \ \ \ \text{as\ \ }x\rightarrow0^{+}\ \ \ \ ,\ \ a\in\left(
0,2\right) \label{G1E2}\\
\bar{h}\left(  x\right)   &  \sim-\frac{2a}{\kappa\Gamma\left(  -a\right)
}\frac{1}{x^{a}}\ \ \ \ \text{as\ \ }x\rightarrow\infty\ \ \ ,\ \ a\in\left(
0,1\right)  \cup\left(  1,2\right) \nonumber
\end{align}%
\begin{align}
\bar{u}\left(  x\right)   &  \sim\frac{2a\kappa}{\Gamma\left(  a\right)
\left(  a+1\right)  }x^{a}\ \ \ \ \text{as\ \ }x\rightarrow0^{+}%
\ \ \ \ ,\ \ a\in\left(  0,2\right) \label{G1E3}\\
\bar{u}\left(  x\right)   &  \sim-\frac{2a}{\kappa\Gamma\left(  -a\right)
\left(  1-a\right)  }\frac{1}{x^{a}}\ \ \ \ \text{as\ \ }x\rightarrow
\infty\ \ \ ,\ \ a\in\left(  0,1\right) \nonumber
\end{align}

If $a=1:$%
\begin{equation}
\bar{v}\left(  x\right)  =2e^{-\kappa x}\ \ ,\ \ \bar{h}\left(  x\right)
=2\left(  \kappa x\right)  e^{-\kappa x}\ \ ,\ \ \bar{u}\left(  x\right)
=\frac{2}{\kappa x}\left(  1-\left(  1+\kappa x\right)  e^{-\kappa x}\right)
\label{G1E4}%
\end{equation}

\end{theorem}

\begin{remark}
The solutions with different values of $\kappa$ are essentially the same up to
rescaling. Notice also that (\ref{F7E7}) implies that $\bar{v}\left(
x\right)  $ becomes negative for $x$ sufficiently large if $a>1.$
\end{remark}

\begin{remark}
Some comments about the choice of the range of values of $a$ are in order. We
need the restriction $0<a<2$ to avoid the singular points $\left(
-\kappa\right)  ^{\frac{1}{a}}$ crossing the imaginary axis. It would also be
possible to include the value $a=1$ in the second formula of (\ref{G1E1}) and
(\ref{G1E2}) since the right-hand side vanishes. However this
would result in asymptotic formulas like $f\sim0 $ that do not have a precise
meaning. Finally the constraint $a<1$ in the second formula of (\ref{G1E3}) is
strictly needed, since the asymptotics of $\bar{u}\left(  x\right)  $ for
$a>1$ is proportional to $\frac{1}{x}.$
\end{remark}

\begin{remark}
It is interesting to note that  the solutions described in Theorem
\ref{solInt} for $0<a<1 $ are just the self-similar solutions with
algebraic decay that have been obtained for the coagulation equation with
constant kernel in \cite{MP1}.  This correspondence can be seen as
follows. The solutions we obtained in Theorem \ref{solInt} satisfy
\[
\bar{v}\left(  x\right)  =\left(  1-a\right)  +\int_{x}^{\infty}\bar{g}\left(
z\right)  dz\ \ ,\ \ \bar{u}\left(  x\right)  =\frac{1}{x}\int_{0}^{x}y\bar
{g}\left(  y\right)  dy
\]
with $\bar{g}\left(  x\right)  =\frac{\bar{h}\left(  x\right)  }{x}.$ Using
this formula to eliminate $\bar{u},\ \bar{v}$ from (\ref{S8E1}) we obtain
\begin{equation}
x^{2}\bar{g}\left(  x\right)  =\left(  1-a\right)  \int_{0}^{x}y\bar{g}\left(
y\right)  dy+\int_{0}^{x}y\bar{g}\left(  y\right)  dy\int_{x-y}^{\infty}%
\bar{g}\left(  z\right)  dz \,.\label{A1}%
\end{equation}

Differentiating (\ref{A1}) we obtain the equation that is satisfied by the
self-similar solutions of
\begin{equation}
\label{A2}\partial_{\tau}F=Q\left[  F\right]  -\frac{\left(  1-a\right)
}{\tau}F
\end{equation}
having the form $f\left(  \xi,\tau\right)  =\frac{1}{\tau^{2}}\bar{g}\left(
\frac{\xi}{\tau}\right)  $ with $Q\left[  \cdot\right]  $ as in (\ref{S1E1}%
)$.$ Equation (\ref{A2}) can be transformed into (\ref{S1E1}) using the change
of variables 
$F\left(  \xi,\tau\right)  =\left(  \tau\right)  ^{-\left(
1-a\right)  }f\left(  \xi,t\right) $ and $t=\frac{\tau^{a}}{a}$ and
correspondingly the solutions obtained in Theorem \ref{solInt} are transformed
into those obtained in \cite{MP1} if $0<a\leq1.$
\end{remark}

\begin{proof}
Integrating the first equation in (\ref{S8E2}) we arrive at
\begin{equation}
\bar{u}\left(  x\right)  =\frac{1}{x}\int_{0}^{x}\bar{h}\left(  z\right)  dz\,.
\label{S8E2a}%
\end{equation}
Using the second equation in (\ref{S8E2}) we obtain
\begin{align}
\int_{x-y}^{x}\frac{\bar{h}\left(  z\right)  }{z}dz  &  =-\bar{v}\left(
x\right)  +\bar{v}\left(  x-y\right)\,, \label{S8E3}\\
\int_{0}^{x}\bar{h}\left(  z\right)  dz  &  =-\int_{0}^{x}z\bar{v}_{z}\left(
z\right)  =-x\bar{v}\left(  x\right)  +\int_{0}^{x}\bar{v}\left(  z\right)  dz\,,
\label{S8E4}%
\end{align}
assuming that all the integrals involved are convergent.  Therefore, using
(\ref{S8E2a}) and (\ref{S8E4})
we find
\begin{equation}
\bar{u}\left(  x\right)  =-\bar{v}\left(  x\right)  +\frac{1}{x}\int_{0}%
^{x}\bar{v}\left(  z\right)  dz \,.\label{S8E5}%
\end{equation}

Eliminating $\bar{h}$ and $\bar{u}$ from (\ref{S8E1}), the second equation in
(\ref{S8E2}), as well as (\ref{S8E2a}), (\ref{S8E3}), (\ref{S8E4}),
implies
\[%
\begin{split}
-x\bar{v}_{x}\left(  x\right)   &  =-\left(  \bar{v}\left(  x\right)  \right)
^{2}+\frac{\bar{v}\left(  x\right)  }{x}\int_{0}^{x}\bar{v}\left(  z\right)
dz\\
&  \quad-\frac{1}{x}\int_{0}^{x}y\bar{v}_{y}\left(  y\right)  \bar{v}\left(
x-y\right)  dy+\frac{\bar{v}\left(  x\right)  }{x}\int_{0}^{x}y\bar{v}%
_{y}\left(  y\right)  dy.
\end{split}
\]

Integrating by parts in the last equation we obtain
\begin{equation}
x^{2}\bar{v}_{x}\left(  x\right)  =\int_{0}^{x}y\bar{v}_{y}\left(  y\right)
\bar{v}\left(  x-y\right)  dy \,.\label{S8E6}%
\end{equation}

Equation (\ref{S8E6}) is a convolution equation. In order to solve it we
introduce the Laplace transform of $\bar{v}\left(  x\right)  $
\begin{equation}
w\left(  \zeta\right)  =\int_{0}^{\infty}\bar{v}\left(  x\right)  e^{-\zeta
x}dx\,. \label{S8E7}%
\end{equation}
Then (\ref{S8E6}) becomes
$\frac{d^{2}}{d\zeta^{2}}\big( \zeta w \big)   +w
\frac{d}{d\zeta}\big( \zeta w \big)  =0$. 
This equation can be integrated explicitly. Its only solution that does not
have singularities and takes real values along the line $\zeta\in
\mathbb{R}^{+}$ is
\begin{equation}
w\left(  \zeta\right)  =\frac{1}{\zeta}\left[  1+a\left(  \frac{\zeta
^{a}-\kappa}{\zeta^{a}+\kappa}\right)  \right] \,. \label{S8E8}%
\end{equation}
Inverting the Laplace transform we obtain%
\begin{equation}
\bar{v}\left(  x\right)  =\frac{1}{2\pi i}\int_{\gamma}\left[  1+a\left(
\frac{\zeta^{a}-\kappa}{\zeta^{a}+\kappa}\right)  \right]  \frac{e^{\zeta x}%
}{\zeta}d\zeta\,. \label{S8E9}%
\end{equation}

The change of variables $\zeta x=\hat{\zeta}$ transforms this formula into
(\ref{F7E5}). The choice of the contour of integration ensures the exponential
convergence of the integral as well as the absence of singularities of the
integrand along the contour of integration. Taking the limits $x\rightarrow
0^{+}$ and $x\rightarrow\infty$ we obtain (\ref{F7E7}).

The negativity of $\bar{v}_{x}$ for $0<a<1$ can be proved by differentiating
(\ref{S8E9}) and deforming the contour of integration to a double half-line
following $\mathbb{R}^{-}$ in a positive and negative direction with
$\arg\left(  \zeta\right)  =-\pi$ and $\arg\left(  \zeta\right)  =\pi$
respectively. Then
\[
\bar{v}_{x}\left(  x\right)  =-\frac{2a\kappa\sin\left(  a\pi\right)  }{\pi
}\int_{0}^{\infty}\frac{e^{-rx}r^{a}dr}{\left\vert \kappa+e^{a\pi i}%
r^{a}\right\vert ^{2}}\ \ ,\ \ 0<a<1\,.
\]

The asymptotics in (\ref{G1E1}) can be obtained rewriting (\ref{S8E9}) as
\[
\bar{v}\left(  x\right)  =\left(  1+a\right)  -\frac{\kappa a}{\pi i}%
\int_{\gamma}\left(  \frac{1}{\zeta^{a}+\kappa}\right)  \frac{e^{\zeta x}%
}{\zeta}d\zeta
\]
and%
\[
\bar{v}\left(  x\right)  =\left(  1-a\right)  +\frac{a}{\pi i}\int_{\gamma
}\left(  \frac{\zeta^{a}}{\zeta^{a}+\kappa}\right)  \frac{e^{\zeta x}}{\zeta
}d\zeta\,.
\]
Differentiating these formulas we obtain
\[
\bar{v}_{x}\left(  x\right)  =-\frac{\kappa a}{\pi i}\int_{\gamma}%
\frac{e^{\zeta x}}{\zeta^{a}+\kappa}d\zeta\ \ ,\ \ \bar{v}_{x}\left(
x\right)  =\frac{a}{\pi i}\int_{\gamma}\left(  \frac{\zeta^{a}}{\zeta
^{a}+\kappa}\right)  e^{\zeta x}d\zeta
\]
and using the change of variables $\zeta x=\hat{\zeta}$ in both integrals it 
follows that
\[
\bar{v}_{x}\left(  x\right)  =-\frac{\kappa a}{\pi i}\frac{1}{x^{1-a}}%
\int_{\gamma}\frac{e^{\zeta}}{\zeta^{a}+\kappa x^{a}}d\zeta\ \ ,\ \ \bar
{v}_{x}\left(  x\right)  =\frac{a}{\pi i}\frac{1}{x}\int_{\gamma}\left(
\frac{\zeta^{a}}{\zeta^{a}+\kappa x^{a}}\right)  e^{\zeta}d\zeta\,.
\]

Taking the limit $x\rightarrow0^{+}$ in the first formula and $x\rightarrow
\infty$ in the second and using the fact that
\[
\frac{1}{\pi i}\int_{\gamma}\frac{e^{\zeta}}{\zeta^{a}}d\zeta=\frac{2}%
{\Gamma\left(  a\right)  }\ \ \ \ \ ,\ \ \ \ \frac{1}{\pi i}\int_{\gamma}%
\zeta^{a}e^{\zeta}d\zeta=\frac{2}{\Gamma\left(  -a\right)  }
\]
we obtain (\ref{G1E1}). Using then (\ref{S8E2}) we obtain (\ref{G1E2}),
(\ref{G1E3}).

Formula (\ref{G1E4}) follows by integrating by residues.
\end{proof}

We can reformulate the results in Theorem \ref{solInt} in terms of the
original functions $\left(  \mathcal{H}\left(  X\right)  ,\mathcal{U}\left(
X\right)  ,\mathcal{V}\left(  X\right)  \right)  .$ In the following we choose $\kappa=1$
as a suitable normalization.

\begin{theorem}
\label{ThIntAs}For any $0<a<2$ there exists a solution of (\ref{S7E7}),
(\ref{S7E9}) satisfying $\mathcal{H}(-\infty)=\mathcal{H}(\infty)=
\mathcal{U}(-\infty)=\mathcal{U}(\infty)=0$. It is given by
\begin{equation}
\mathcal{V}_{a}\left(  X\right)  -1=\frac{a}{2\pi i}\int_{\gamma}\left(
\frac{\zeta^{a}-e^{aX}}{\zeta^{a}+e^{aX}}\right)  \frac{e^{\zeta}}{\zeta
}d\zeta\label{T1E1}%
\end{equation}%
\begin{equation}
\mathcal{H}_{a}\left(  X\right)  =-\frac{d\mathcal{V}_{a}\left(  X\right)
}{dX}\ \ ,\ \ \mathcal{U}_{a}\left(  X\right)  =-\int_{-\infty}^{X}e^{\left(
Z-X\right)  }\frac{d\mathcal{V}_{a}\left(  Z\right)  }{dZ}dZ \label{T1E2}%
\end{equation}%
\begin{equation}
\lim_{X\rightarrow-\infty}\mathcal{V}_{a}\left(  X\right)
=1+a\ \ \ ,\ \ \ \lim_{X\rightarrow\infty}\mathcal{V}_{a}\left(  X\right)
=1-a\,. \label{T1E3}%
\end{equation}

If $0<a<1$ we have $\frac{d\mathcal{V}_{a}\left(  X\right)  }{dX}%
<0,\ \mathcal{H}_{a}\left(  X\right)  >0,\ \mathcal{U}_{a}\left(  X\right)
>0$ for any $X\in\mathbb{R}$. We have the asymptotics
\begin{align}
\frac{d\mathcal{V}_{a}\left(  X\right)  }{dX}  &  \sim-\frac{2a}{\Gamma\left(
a\right)  }e^{aX}\ \ \text{as\ \ }X\rightarrow-\infty,\ \ a\in\left(
0,2\right) \label{T1E4}\\
\frac{d\mathcal{V}_{a}\left(  X\right)  }{dX}  &  \sim\frac{2a}{\Gamma\left(
-a\right)  }e^{-aX}\ \ \text{as\ \ }X\rightarrow\infty\ \ \ ,\ \ a\in\left(
0,1\right)  \cup\left(  1,2\right) \nonumber
\end{align}%
\begin{align}
\mathcal{H}_{a}\left(  X\right)   &  \sim\frac{2a}{\Gamma\left(  a\right)
}e^{aX}\ \ \ \ \text{as\ \ }X\rightarrow-\infty\ \ \ \ ,\ \ a\in\left(
0,2\right) \label{T1E5}\\
\mathcal{H}_{a}\left(  X\right)   &  \sim-\frac{2a}{\Gamma\left(  -a\right)
}e^{-aX}\ \ \ \ \text{as\ \ }X\rightarrow\infty\ \ \ ,\ \ a\in\left(
0,1\right)  \cup\left(  1,2\right) \nonumber
\end{align}%
\begin{align}
\mathcal{U}_{a}\left(  X\right)   &  \sim\frac{2a}{\Gamma\left(  a\right)
\left(  a+1\right)  }e^{aX}\ \ \ \ \text{as\ \ }X\rightarrow-\infty
\ \ \ \ ,\ \ a\in\left(  0,2\right) \label{T1E6}\\
\mathcal{U}_{a}\left(  X\right)   &  \sim-\frac{2a}{\Gamma\left(  -a\right)
\left(  1-a\right)  }e^{-aX}\ \ \ \ \text{as\ \ }X\rightarrow\infty
\ \ \ ,\ \ a\in\left(  0,1\right) \nonumber
\end{align}

If $a=1:$%
\begin{equation}
\mathcal{V}_{a}\left(  X\right)  =2e^{-e^{X}} ,\, \mathcal{H}_{a}\left(
X\right)  =2 e^{X} e^{-e^{X}} ,\, \mathcal{U}_{a}\left(  X\right)
=2e^{-X}\left(  1-\left(  1+\kappa e^{X}\right)  e^{-e^{X}}\right)
\label{T1E7}%
\end{equation}

\end{theorem}

\subsection{Solution of (\ref{S7E10})}
\label{Ss.odeanalysis}

Equation (\ref{S7E10}) that approximates the behaviour of the solutions of
(\ref{S3E10})-(\ref{S3E13}) if $\lambda E\approx1$ and $U$ is of order one,
can be solved explicitly since
\begin{equation}
\hat{E}=\lambda\left(  \log\left(  U\right)  -U\right)  -V+\log\left(
V\right)  \label{S9E1}%
\end{equation}
is preserved along trajectories.
The solutions of (\ref{S7E10}) will be used to match the solutions of
(\ref{S7E7}), (\ref{S7E9}) obtained in the previous Section for $1\ll U\ll \frac
{1}{\lambda}.$ For this range of values $V$ is almost constant if
$\lambda\rightarrow0.$ Moreover, since we are away from the ODE regime
described in Section \ref{ODE} we can assume that $\lambda E$ in (\ref{S6E6})
is of order one and therefore $\hat{E}$ is of order one,
 whence also $\left\vert V-1\right\vert $ is of order one.

Suppose that we consider a matching region (to be made precise later) where 
$\hat{E}$ is of order one,  $V$ is almost constant and satisfies $V<1$. 
Then
 the first equation in (\ref{S7E10}) implies that $U$ decreases
exponentially. More preciesely, let us assume that $U=1$ for $X=X^{-}.$ Then, the asymptotics
of $U,\ V$ are
\begin{equation}
V\sim\left(  1-a^{-}\right)  \ \ ,\ \ \ \ U\sim\exp\left(  -a_{-}\left(
X-X^{-}\right)  \right)  \label{S9E2}%
\end{equation}
for some $a^{-}\in\left(  0,1\right)  $.
The asymptotics (\ref{S9E2}) will be shown to match with suitable solutions
among the ones described in Theorem \ref{solInt}. Moreover, notice that
(\ref{S7E10}) implies that these asymptotics are valid as long as $U$ remains
larger than some small number $\varepsilon_{0}$ that we can assume to be of
order one, although small. The exponential decay of $U$ in (\ref{S9E2})
implies that the transition between $U=1$ and $U=\varepsilon_{0}$ takes place
on a length scale $\left(  X-X^{-}\right)  $ of order one.

We now describe  the solution when $U$ is less than $\varepsilon_0$. For this
purpose
let us define by $a_{-}$ and $a_{+}$ the two roots of
the equation $\hat{E}=-V+\log\left(
V\right) $. Then, 
 in the limit $\lambda\rightarrow0^{+}$, the curves given by
(\ref{S9E1}) in the plane $\left(  U,V\right)  $
 are approximately
the two horizontal lines $\left\{  V=1-a_{-}\right\} $ and $\left\{
V=1+a_{+}\right\}$ 
 connected by  a  curve that is approximately vertical. Along this connecting curve 
 $V$ changes an amount of order one and
 $\lambda\log\left(  U\right)  $ continues to be of order one since $\hat E$ is
preserved.
When $U$ is less than $\varepsilon_{0}$ we can approximate the second
equation in (\ref{S7E10}) as $\frac{dV}{dX}=\lambda\left(  1+O\left(
\varepsilon_{0}\right)  \right)  V.$ Then $V\left(  X\right)  $ increases
exponentially. The first equation in (\ref{S7E10}) 
 indicates that for the
solutions of (\ref{S7E10}) under consideration, $U$ remains small while $V$
changes from $\left(  1-a_{-}\right)  $ to$\ \left(  1+a_{+}\right)  .$ Then
we can use the approximation
\begin{equation}
V=\left(  1-a_{-}\right)  \exp\left(  \lambda\left(  1+O\left(  \varepsilon
_{0}\right)  \right)  \left(  X-X^{-}-\ell_{trans}^{-}\right)  \right)
\label{S9E3}%
\end{equation}
where $\ell_{trans}^{-}$ is the range of $X$ that it takes for $U$ to decrease
from $U=1$ to $U=\varepsilon_{0}.$ We recall that $\ell_{trans}^{-}$ is of
order one.

Notice that the structure of the curve (\ref{S9E1}) implies that eventually
$U$ becomes again of order $\varepsilon_{0}$ with $V$ close to $\left(
1+a_{+}\right)  .$ It then follows from (\ref{S7E10}) that for small
$\lambda,$ $U$ reaches again the value $1$ at  $X^{+}$ after
an additional length
$\ell_{trans}^{+}$ of order one. Notice that, using (\ref{S9E3}), we have
\begin{equation}
V=\left(  1-a_{-}\right)  \exp\left(  \lambda\left(  1+O\left(  \varepsilon
_{0}\right)  \right)  \left(  X^{+}-X^{-}-\ell_{trans}^{+}-\ell_{trans}%
^{-}\right)  \right)  \label{S9E4}%
\end{equation}
and $U$ has the behaviour, for $\varepsilon_{0}\leq U\ll 
\frac{1}{\lambda}$,
\begin{equation}
U\sim\exp\left(  a_{+}\left(  X-X^{+}\right)  \right)\,.  \label{S9E5}%
\end{equation}
We need to estimate the relation between $a_{-}$ and $a_{+}.$ From
(\ref{S9E1}) we deduce that to leading order
\begin{equation}
\log\left(  1-a_{-}\right)  +a_{-}=\log\left(  1+a_{+}\right)  -a_{+}\,.
\label{S9E6}%
\end{equation}

On the other hand (\ref{S9E4}) allows us to obtain an approximation for the
transition length $\left(  X^{+}-X^{-}\right)  .$ Indeed, we have
\[
\left(  1+a_{+}\right)  =\left(  1-a_{-}\right)  \exp\left(  \lambda\left(
1+O\left(  \varepsilon_{0}\right)  \right)  \left(  X^{+}-X^{-}-\ell
_{trans}^{+}-\ell_{trans}^{-}\right)  \right)
\]
whence, since $\ell_{trans}^{+},\ \ell_{trans}^{-}$ are of order one and
$\varepsilon_{0}$ can be made arbitrarily small
\begin{equation}
\left(  X^{+}-X^{-}\right)  \sim\frac{1}{\lambda}\log\left(  \frac{1+a_{+}%
}{1-a_{-}}\right)  \label{S9E7}%
\end{equation}
as $\lambda\rightarrow0^{+}.$

We have described the transition of the solution $\left(  U,V,H\right)  $ of
(\ref{S3E10})-(\ref{S3E13}) for small $\lambda$ and $U\ll \frac{1}{\lambda}.$
The following result will play a crucial role in describing the solutions,
because it will show that the amplitude of the oscillations increases with
$X.$

\begin{lemma}
\label{L1}Suppose that $a_{-}>0,$ $a_{+}>0$ satisfy (\ref{S9E6}). Then
$a_{+}>a_{-}$.
\end{lemma}

\begin{proof}
This follows from the inequality $\log\left(  1+x\right)  -x>\log\left(
1-x\right)  +x$ that is just a consequence of
$\log\left(  1+x\right)  -\log\left(  1-x\right)  =\int_{0}^{x}\frac
{2dr}{\left(  1-r^{2}\right)  }>2x$.
\end{proof}

\subsection{Coupling of the regimes}

\label{asymptPeaks}

\subsubsection{Description of the iterative procedure\label{iteration}}

We now describe the behaviour of the solutions of (\ref{S3E10})-(\ref{S3E13})
for $\lambda\rightarrow0$ and $\lambda E\approx1,$ with $E$ as in
(\ref{S6E6}). It turns out that these solutions can be described for this
range of values by a sequence of intervals  where the solutions can be
approximated alternately by solutions of (\ref{S7E7}), (\ref{S7E9})
or by solutions of (\ref{S7E10}) with both types of regions connected with
 suitable matching regions.

More precisely, suppose that $X_{n}^{+}$ is a point where
$U\left(  X_{n}^{+}\right)  =1\ ,\ \ V\left(  X_{n}^{+}\right)  =1+a_{n}$.
Notice that at such a point we can expect to be able to use the approximation
(\ref{S7E10}). This will be seen matching the ODE-Integrodifferential equation
regime with the  ODE regime described in Section
\ref{ODELolkaVolterra}. We will assume for the moment that this approximation
is valid. Then $H$ is also of order one and we can approximate it to the
leading order as $H=UV$ (cf. Section \ref{Ss.ode}). We can also use the
approximation (\ref{S9E5}) with $a_{+}=a_{n}$ and  approximate $V$ by a
constant. Then we obtain for $\left\vert X-X_{n}^{+}\right\vert $ of order one that
\begin{equation}
U\sim\exp\left(  a_{n}\left(  X-X_{n}^{+}\right)  \right)  \ \ ,\ \ V\sim
1+a_{n}\ \ ,\ \ H\sim\left(  1+a_{n}\right)  \exp\left(  a_{n}\left(
X-X_{n}^{+}\right)  \right)\,.  \label{T2E1}%
\end{equation}

We can now match the asymptotics (\ref{T2E1}) with those obtained for the
solutions of (\ref{S7E7}), (\ref{S7E9}). We use (\ref{S7E1}) to obtain the
 matching condition
\begin{equation}%
\begin{split}
\mathcal{U}\left(  X\right)  \sim\lambda\exp\left(  a_{n}\left(  X-X_{n}%
^{+}\right)  \right)  \   & ,\ \ \mathcal{V}\left(  X\right)  \sim\left(
1+a_{n}\right)\,, \\
\mathcal{H}\left(  X\right)  &  \sim\lambda\left(  1+a_{n}\right)  \exp\left(
a_{n}\left(  X-X_{n}^{+}\right)  \right)\,.
\end{split}
\label{T2E2}%
\end{equation}

This behaviour must be matched with the one for the solutions of (\ref{S7E7}),
(\ref{S7E9}) obtained in Theorem \ref{ThIntAs}. More precisely we will match
these behaviours with the ones of
\[
\Phi_{n}\left(  X\right)  =\left(  \mathcal{U}_{a_n}\left(  X-X_{n}^{0}\right)
,\mathcal{V}_{a_n}\left(  X-X_{n}^{0}\right)  ,\mathcal{H}_{a_n}\left(
X-X_{n}^{0}\right)  \right)
\]
for a suitable choice of
$a_n$ and $ X_{n}^{0}.$ The asymptotics of $\mathcal{V}\left(  X\right)  $ combined
with the one in (\ref{T1E3}) yield $a=a_{n},$ since in the matching region we
expect to have $\left(  X-X_{n}^{0}\right)  \rightarrow-\infty.$ On the other
hand (\ref{T1E5}), (\ref{T1E6}) give the choice
\begin{equation}
X_{n}^{0}=X_{n}^{+}+\frac{1}{a_n}\log\left(  \frac{2a_{n}}{\left(  1+a_{n}\right)
\Gamma\left(  a_{n}\right)  }\frac{1}{\lambda}\right)\,.  \label{T2E3}%
\end{equation}
With these choices of $a_{n}$ and $ X_{n}^{0}$ we obtain, using (\ref{T1E3}),
(\ref{T1E5}), (\ref{T1E6}) and (\ref{T2E2}), a matching to the first order
between (\ref{T2E1}) and $\Phi_{n}\left(  X\right)  $ in the common region of
validity where
$\left(  X-X_{n}^{+}\right)  \gg 1$ and $ \left(  X-X_{n}^{0}\right)  \ll 1$.

We can then use the approximation $\Phi_{n}\left(  X\right)  $ to describe the
solutions of (\ref{S3E10})-(\ref{S3E13}) for small $\lambda$ and $\left(
X-X_{n}^{0}\right)  $ of order one. This approximation breaks down for
$\left(  X-X_{n}^{0}\right)  $ sufficiently large. Indeed, (\ref{T1E5}) and
(\ref{T1E6}) imply that $\mathcal{U}_{a}\left(  X-X_{n}^{0}\right)  $ and
$\mathcal{H}_{a}\left(  X-X_{n}^{0}\right)  $ converge exponentially to zero
as $\left(  X-X_{n}^{0}\right)  \rightarrow\infty.$ However, the approximation
(\ref{S7E7}), (\ref{S7E9}) is only valid if $\left\vert \left(  \mathcal{U}%
\left(  X\right)  ,\mathcal{H}\left(  X\right)  \right)  \right\vert
\gg \lambda.$ In the region where $\left\vert \left(  \mathcal{U}\left(
X\right)  ,\mathcal{H}\left(  X\right)  \right)  \right\vert $ becomes of
order $\lambda$ we must use again the approximation (\ref{S7E10}). In
particular this region can be described using the analysis in Section
\ref{Ss.odeanalysis}. The asymptotics of the solutions for $U$ of order one is
as in (\ref{S9E2}) for suitable choices of $a_{-},\ X^{-}.$ More precisely,
since this matching is made for $\left(  X-X_{n}^{0}\right)  \gg 1$ we can use
the asymptotics (\ref{T1E3}), (\ref{T1E5}), (\ref{T1E6}) for $X\rightarrow
\infty.$ We then obtain, using also the rescaling (\ref{S7E1}), the matching
conditions
\[%
\begin{split}
U\left(  X\right)  \sim-\frac{2a_{n}}{\Gamma\left(  -a_{n}\right)  \left(
1-a_{n}\right)  \lambda}e^{-a_{n}\left(  X-X_{n}^{0}\right)  }\ , &
\ \ \ V\left(  X\right)  \sim\left(  1-a_{n}\right) \\
H\left(  X\right)     \sim-\frac{2a_{n}}{\Gamma\left(  -a_{n}\right)
\lambda}e^{-a_{n}\left(  X-X_{n}^{0}\right)  } & \qquad \mbox{ for } 
\left(  X-X_{n}^{0}\right)  \gg 1,\ \ U\gg 1\,.%
\end{split}
\]
We now match these formulas with (\ref{S9E2}) (combined with the approximation
$H=UV$). To this end we must choose
\begin{equation}
a_{-}=a_{n}\  ,\ \ \ X_{n}^{-}=X_{n}^{0}+\frac{1}{a_n}\log\left(  -\frac{2a_{n}%
}{\left(  1-a_{n}\right)  \Gamma\left(  -a_{n}\right)  }\frac{1}{\lambda
}\right) \,. \label{T2E4}%
\end{equation}

We remark that all this analysis will be meaningful only if $0<a_{n}<1.$
Therefore $\Gamma\left(  -a_{n}\right)  <0.$

The region where $U\leq1,\ H\leq1$ can then be described using the analysis in
Section \ref{Ss.odeanalysis}. The conclusion of this analysis is that $\left(
U,V\right)  $ moves close to the point $\left(  1,\left(  1+a_{n+1}\right)
\right)  $ for some suitable $a_{n+1}$ in a characteristic length given by
(\ref{S9E7}). More precisely if we define (cf. (\ref{S9E6}))
\begin{equation}
\log\left(  1-a_{n}\right)  +a_{n}=\log\left(  1+a_{n+1}\right)
-a_{n+1}\ \ ,\ \ a_{n}>0\ \ ,\ \ a_{n+1}>0 \label{T2E5}%
\end{equation}

\begin{equation}
X_{n+1}^{+}=X_{n}^{-}+\frac{1}{\lambda}\log\left(  \frac{1+a_{n+1}}{1-a_{n}%
}\right)  \label{T2E6}%
\end{equation}
we can approximate $U,\ V$ by means of (\ref{T2E1}) with $X_{n}^{+}$ replaced
by $X_{n+1}^{+}.$

Combining (\ref{T2E3}), (\ref{T2E4}) and (\ref{T2E6}) we obtain to the leading
order%
\begin{equation}
X_{n+1}^{+}-X_{n}^{+}\sim\frac{1}{\lambda}\log\left(  \frac{1+a_{n+1}}%
{1-a_{n}}\right)\,.  \label{T2E6a}%
\end{equation}

\begin{figure}[ht!]
\centering{%
\includegraphics[width=.46\textwidth]{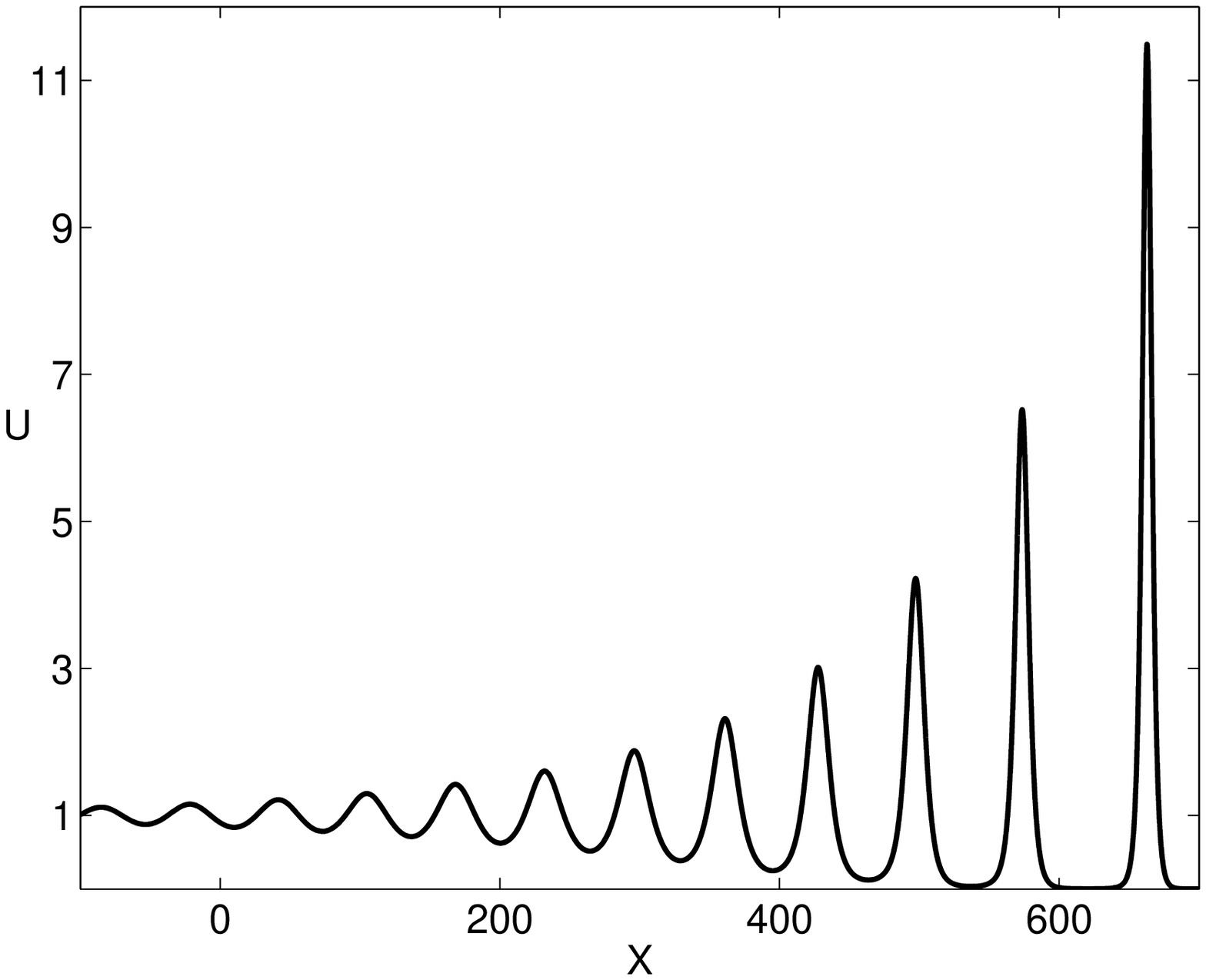}
\includegraphics[width=.46\textwidth]{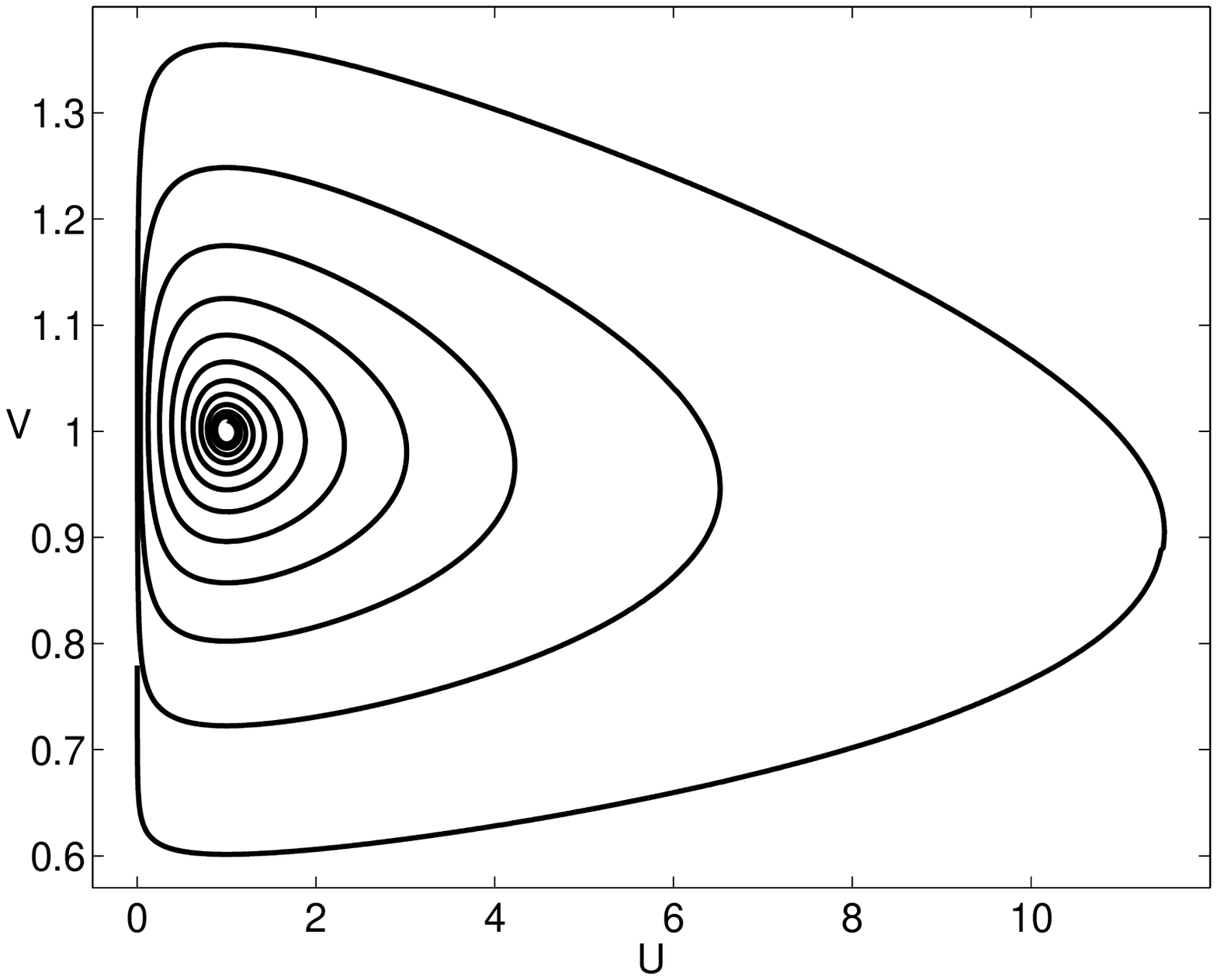}
}%
\caption{Intermediate regime: $U$ (left) and phase plane for $U,V$ (right)}
\label{figure2}
\end{figure}

\subsubsection{Matching with the ODE regime}

We first remark that equation (\ref{S5E6}) agrees
with the approximating equations (\ref{S7E7}), (\ref{S7E9}) and (\ref{S7E10})
if $E\gg 1,\ E\ll \frac{1}{\lambda}$ in their respective regimes of validity.

Indeed, we can rewrite (\ref{S5E6}) using (\ref{S5E5}) as
\begin{equation}
\frac{dU}{dX}=U\left(  V-1\right)  +\lambda U\left(  U-1\right)
\ \ \ ,\ \ \ \frac{dV}{dX}=\lambda\left(  1-U\right)  V \label{T2E7}%
\end{equation}
and these equations yield the same behaviour as (\ref{S7E10}) if $U$ is bounded.

On the other hand, in order to approximate the solutions of (\ref{S7E7}),
(\ref{S7E9}) for $E\gg 1,\ E\ll \frac{1}{\lambda}$ we use the fact that for small
$a,$ the solutions of (\ref{S7E7}), (\ref{S7E9}) in Theorem \ref{ThIntAs} have
the characteristic length scale $\frac{1}{a}.$ Then, for $a$ small, which
corresponds to the matching region indicated above, we have a characteristic
length very large compared with one. Then, as in (\ref{S5E2}) we can approximate $I\left[
\mathcal{H}\right] \sim \lambda \mathcal{H} \mathcal{U}$.  
Since $E\ll \frac{1}{\lambda}$ we have $\mathcal{U}\left(  X\right)  \ll 1$
and we obtain from (\ref{S7E7}) the approximation
$\mathcal{H}\left(  X\right)  =\mathcal{U}\left(  X\right)  \mathcal{V}\left(
X\right)  +\left(  \mathcal{U}\left(  X\right)  \right)  ^{2}\mathcal{V}%
\left(  X\right)$
and plugging this formula into (\ref{S7E9}) and using the rescaling
(\ref{S7E1}) we obtain
\[
\frac{dU}{dX}=U\left(  V-1\right)  +\lambda\left(  U\right)  ^{2}%
V\ \ ,\ \ \frac{dV}{dX}=-\lambda UV-\lambda^{2}\left(  U\right)  ^{2}V
\]
that agrees with (\ref{T2E7}) in the range of values $U\gg 1,\ U\ll \frac
{1}{\lambda}.$

The previous computations show that the equations used in both transition regimes
agree in the intermediate matching regime. Actually it is possible to check
that the main characteristics of the computed solution also agree, as  could
be expected. Notice that the length scale $\left(  X_{n+1}^{+}-X_{n}%
^{+}\right)  $ can be computed using the 
 function $T\left(  E\right)  $ in (\ref{S6E5a}) for
$E\gg 1,\ E\ll \frac{1}{\lambda}$:
\begin{equation}
T\left(  E\right)  \sim2\sqrt{2}\sqrt{E}\,. \ \label{T3E1}%
\end{equation}

In the computation of this integral, we have split the region of integration
into the two subsets $\left(  U_{-}\left(  E,0\right)  ,\varepsilon_{0}\right)
$ and $\left(  \varepsilon_{0},U_{+}\left(  E,0\right)  \right)  $ where
$\varepsilon_{0}>0$ is a small but fixed number. It turns out that the
contribution of the second integral is much smaller than the one due to the
first integrand, exactly in the same manner that the contributions in $\left(
X_{n+1}^{+}-X_{n}^{+}\right)  $ computed in Subsection \ref{iteration} are due
mostly to the region with small values of $U.$ Using the rescaling
(\ref{S5E5}) we then obtain the  approximation 
\begin{equation}
\left(  X_{n+1}^{+}-X_{n}^{+}\right)  \sim\frac{2\sqrt{2}\sqrt{E}}%
{\sqrt{\lambda}}\,. \label{T3E2}%
\end{equation}

On the other hand, we can compute the same length using (\ref{T2E6a}). Taking
into account that in the matching region $a_{n},\ a_{n+1}$ are small, we
obtain from (\ref{T2E5}) that
\begin{equation}
a_{n+1}\sim a_{n}+\frac{2}{3}\left(  a_{n}\right)  ^{2}\,. \label{T3E3}%
\end{equation}
Using  (\ref{T2E6a})  we  then obtain to
 leading order that
\begin{equation}
X_{n+1}^{+}-X_{n}^{+}\sim\frac{\left(  a_{n}+a_{n+1}\right)  }{\lambda}%
\sim\frac{2a_{n}}{\lambda} \,.\label{T3E4}%
\end{equation}
Next, equation (\ref{S6E6}) and the definition of $a_{n}$ by means of $U=1,$ imply
\begin{equation}
a_{n}=\sqrt{2}\sqrt{E}\sqrt{\lambda} \label{T3E4a}%
\end{equation}
and plugging (\ref{T3E4a}) into (\ref{T3E4}) we obtain the sought-for matching with
(\ref{T3E2}).

Finally we obtain the matching for the formula of the change of energy in each
iteration. We have obtained using the ODE approximation that the
change of energy in each cycle is given by (\ref{S6E7}) where in the matching
region $1\ll E\ll \frac{1}{\lambda}$ we can use the approximation (\ref{S6E8a}).
Then
\begin{equation}
E_{n+1}-E_{n}\sim\frac{4\sqrt{2}}{3}E_{n}^{\frac{3}{2}}\,. \label{T3E5}%
\end{equation}
On the other hand, using again $a_{n}=\sqrt{2}\sqrt{E}\sqrt{\lambda}$ and
plugging it in (\ref{T3E3}) we obtain again (\ref{T3E5}) and this yields the
desired matching.

\section{Shooting argument}

\label{S.shooting}

The structure of the solutions of (\ref{S3E10})-(\ref{S3E13}) described in
Section \ref{asymptPeaks} by means of alternate approximate solutions of
(\ref{S7E7})-(\ref{S7E9}) and (\ref{S7E10}) yields several sequences of
numbers $\left\{  a_{n}\right\}  ,\ \left\{  X_{n}^{+}\right\}  ,\ \left\{
X_{n}^{-}\right\}  ,\ \left\{  X_{n}^{0}\right\}  $ that measure the amplitude
of the oscillations, the positions of the points where $U=1,\ V>1,\ U=1,\ V<1$
and the position of the peaks where $U$ and $H$ are of order $\frac{1}%
{\lambda}$ respectively. It is relevant to notice that the sequence $\left\{
a_{n}\right\}  $ is increasing due to Lemma \ref{L1}. The values of $a_{n}$
approach zero in the matching region (see the approximation (\ref{T3E4a}) for
$1\ll E\ll \frac{1}{\lambda}$), but they become of order one for $E\approx\frac
{1}{\lambda}$ with increases in each iteration of order one. Eventually, this
sequence of points becomes larger than or equal to one. 

A detailed study of the solutions of (\ref{S3E10})-(\ref{S3E13}) shows that
there are several possibilities. If $a_{n}$ becomes strictly larger than one,
Theorem \ref{ThIntAs} shows that during the peak regime $V$
becomes strictly negative. In particular $V$ reaches the value $V=0$ at some
finite $X_{\ast}.$ If $V$ remains positive for larger values of $X$, something
that could happen if $a_{n}=1,$ it would be possible to approximate
(\ref{S3E10})-(\ref{S3E13}) by means of (\ref{S7E10}). Therefore $V$ would
increase, and the value of $a_{n}$ would be increased to a new value
$a_{n+1}>1$ at the next iteration. The only possibility of not having such a
behaviour would be with $V\rightarrow0$ as $X\rightarrow\infty.$ Such
approximation to zero cannot be described by the approximate equation
(\ref{S7E10}), because in such a regime the integral terms in (\ref{S3E10}%
)-(\ref{S3E13}) must be taken into account in full detail.

We now recall that for any $K\in\mathbb{C}$ there exists a unique solution of
(\ref{S3E10})-(\ref{S3E13}) satisfying (\ref{S4E0}). The discussion in Section
\ref{Ss.variables} shows that all the solutions of (\ref{S3E10})-(\ref{S3E13})
with such a behaviour can be obtained, up to rescaling, by choosing $K$ in the
interval $\left[  1,\exp\left(  \frac{2\pi\beta\left(  \lambda\right)
}{\alpha\left(  \lambda\right)  }\right)  \right)  .$ A continuity argument
taking into account the behaviour of the trajectories shows that there exists
a value of $K=K_{\ast}$ in this interval for which the corresponding solution
of (\ref{S3E10})-(\ref{S3E13}) satisfies $\lim_{X\rightarrow\infty}V\left(
X\right)  =0.$ This also implies $\lim_{X\rightarrow\infty}H\left(  X\right)
=0$ and $\lim_{X\rightarrow\infty}U\left(  X\right)  =0$ using (\ref{S3E10}),
(\ref{S3E12}). More precisely, the inversion of (\ref{S3E10}) allows us to write
$H\left(  X\right)  $ in terms of $UV.$ Then $\lim_{X\rightarrow\infty
}H\left(  X\right)  =0$ and (\ref{S3E12}) yields $\lim_{X\rightarrow\infty
}U\left(  X\right)  =0.$ 

\section{Asymptotics as $X\rightarrow\infty$}

\label{S.asymptotics}

We finally describe the asymptotics of the solution of (\ref{S3E10}%
)-(\ref{S3E13}) satisfying (\ref{S4E0}) with $K=K_{\ast}$ described in the
previous Section. We have seen that $
\lim_{X\rightarrow\infty}U\left(  X\right)  =\lim_{X\rightarrow\infty}V\left(
X\right)  =\lim_{X\rightarrow\infty}H\left(  X\right)  =0\,.$
To  leading order the behaviour of this trajectory is described by the
solution of (\ref{S7E7}), (\ref{S7E9}) given in Theorem \ref{ThIntAs} with
$a=1.$ The  asymptotics 
\begin{equation}
H\left(  X\right)  \sim\frac{2e^{\left(  X-X_{n}^{0}\right)  }e^{-e^{\left(
X-X_{n}^{0}\right)  }}}{\lambda} \label{T4E1}%
\end{equation}
is valid as long as $H\gg 1,\ \left(
X-X^0_{n}\right)  \gg 1$,
where, by assumption, to the leading order $a_{n}=1.$ It is convenient to
reformulate (\ref{T4E1}) in the original variables (cf. (\ref{S3E8}),
(\ref{S3E9}))
\begin{equation}
h\left(  x\right)  \sim\frac{2}{\lambda}\left(  \frac{x}{x_{n}}\right)
\exp\left(  -\left(  \frac{x}{x_{n}}\right)  \right)  \label{T4E2}%
\end{equation}
that is valid for $\frac{x}{x_{n}}\gg 1,\ h\gg 1.$

We can match (\ref{T4E2}) with the following exponential asymptotics for the
solutions of (\ref{S2E5}) as $x\rightarrow\infty$:
\begin{equation}
h\left(  x\right)  \sim\frac{2}{\lambda}
\left(  \kappa x\right)  \exp\left(  -\kappa x\right)
\label{T4E3}%
\end{equation}

The asymptotics (\ref{T4E2}), (\ref{T4E3}) match for small $\lambda$ if
$\kappa=\frac{1}{x_{n}}.$ This gives the desired exponential decay for the
obtained solution.

\section{Self-consistency of the approximations of
$I\left[  H\right]$}
\label{S.consistency}

We now show that the solution we obtained is self-consistent with
the approximations made for the integral operator $I\left[  H\right]  \left(
X\right)  $ defined in (\ref{S3E13}). We have made two main approximations.
The first one is (\ref{E1}), that approximates $I\left[  H\right]  $ by an
operator with a simple dependence on $\lambda.$ The second approximation is
(\ref{S5E2}) and it allows us to replace the integral operator by a much simpler
local operator.

Concerning the validity of (\ref{E1}) we remark that its precise meaning is that the error
made in the approximation is smaller than the right-hand side. In order to
check its validity we distinguish between the regions where $H$ is of order
$\frac{1}{\lambda}$ and the regions where $H$ is smaller than that quantity.

Notice that for any region we can expect the errors made in the approximation
to be of order
\begin{align*}
R_{1}  & =\lambda^{2}\int_{-\infty}^{X}dYe^{\left(  Y-X\right)  }\left(
X-Y\right)  H\left(  Y\right)  \int_{\log\left(  e^{X}-e^{Y}\right)  }%
^{X}dZH\left(  Z\right) \,, \\
R_{2}  & =\lambda^{2}\int_{-\infty}^{X}dYe^{\left(  Y-X\right)  }H\left(
Y\right)  \int_{\log\left(  e^{X}-e^{Y}\right)  }^{X}dZ\left(  X-Z\right)
H\left(  Z\right)
\end{align*}
and we can expect these two terms to be very small compared with the
right-hand side of (\ref{E1}). Indeed, in the case of $R_{1}$, the
contribution due to the region $\lambda\left(  X-Y\right)  \leq\varepsilon
_{0}$ is small compared with the term in (\ref{E1}) if $\varepsilon_{0}$ is
small. On the other hand, if $\lambda\left(  X-Y\right)  >\varepsilon_{0}$ we
can use the smallness of the exponential factor $e^{\left(  Y-X\right)  }$ to
obtain estimates for the corresponding terms in $R_{1}$ as $C\lambda
^{2}\left\Vert H\right\Vert _{\infty}^{2}e^{-\frac{\varepsilon_{0}}{\lambda}}$
and since $\left\Vert H\right\Vert _{\infty}\leq C/\lambda$ it then follows
that this contribution is exponentially small. Concerning $R_{2}$ we can argue
similarly for $\lambda\left(  X-Z\right)  $ smaller than $\varepsilon_{0}$. If
$\lambda\left(  X-Z\right)  >\varepsilon_{0}$, since $Z\geq X+\log\left(
1-e^{Y-X}\right)  $ it follows that the size of the region of integration can
be bounded by $Ce^{-\frac{\varepsilon_{0}}{\lambda}}$ and therefore, it gives
also a very small contribution.

Concerning the approximations of $I\left[  H\right]  $ by means of local terms
that have been made in the derivation of the ODE approximations (\ref{S5E6})
or (\ref{S7E10}) two main ingredients are required. In the derivation of
(\ref{S5E6}) we have used just the fact that $H\left(  Y\right)  $ has
significant changes over distances much longer than one. Since the order of
magnitude of $H$ is roughly the same for the range of values described by
means of (\ref{S5E6}) no special care is required to estimate the values of
$H\left(  Y\right)  $ with $X-Y\gg 1.$ On the other hand, in the case of the
approximation (\ref{S7E10}) we use the fact that $I\left[  H\right]  $ is a
corrective term that is completely ignored in (\ref{S7E10}). Notice that
$H\left(  Y\right)  $ takes values for $Y<X$ much larger than the ones of
$H\left(  X\right)  $ in the region where the approximation (\ref{S7E10}) is
used. However, due to the exponential factor $e^{\left(  Y-X\right)  }$ as
well as the exponential decay of $H\left(  Y\right)  $ in the region described
by the integro-differential equation (cf. Subsection \ref{iteration}) the
corresponding contribution of such  values of $H\left(  Y\right)  $ is
negligible.

\appendix
\section{Uniform estimates for the change of energy}

\label{A.energy}

During the part of the dynamics dominated by a perturbation of a conservative ODE it
has been assumed that the energy $E$ does not change  significantly
during each cycle of length $T\left(  E\right)  .$ This is the
justification for the adiabatic approximation and we will check now that, indeed,
 the adiabatic approximation yields such smallness for the variation of
the energy during the trajectory and thus the assumption is self-consistent.

Let us denote by $\xi_{0}\in\mathbb{R}$ a value where $U\left(  \xi
_{0}\right)  =1,\ V\left(  \xi_{0}\right)  >1.$ Our goal is to estimate this
quantity for $\xi_{0}\leq\xi\leq\xi_{0}+T\left(  E\left(  \xi_{0}\right)
\right)  ,\ $where $\xi_{0}+T\left(  E\left(  \xi_{0}\right)  \right)  $ is
the next value of $\xi$ where $U\left(  \xi\right)  =1$ and $V\left(
\xi\right)  >1$ again. Due to the invariance of (\ref{S5E6}) we can assume
without loss of generality that $\xi_{0}=0.$ Due to (\ref{S6E5}) and assuming
the adiabatic approximation we obtain the following approximation for the
change of the energy along a trajectory:
\begin{equation}
\frac{E\left(  \xi\right)  -E\left(  0\right)  }{\sqrt{\lambda}}=\int_{0}%
^{\xi}\left[  \left(  U_{0}\left(  s\right)  -1\right)  ^{2}+\left(
\omega_{0}\left(  s\right)  \right)  ^{2}\left(  1-U_{0}\left(  s\right)
\right)  \right]  ds=\sigma\left(  \xi\right)\,.  \label{A1E1}%
\end{equation}
Since
\[
\frac{d\sigma\left(  \xi\right)  }{d\xi}=\left(  U_{0}\left(  \xi\right)
-1\right)  \left[  \left(  U_{0}\left(  \xi\right)  -1\right)  -\left(
\omega_{0}\left(  \xi\right)  \right)  ^{2}\right]
\]
and given the form of the curves $\left(  U_{0},\omega_{0}\right)  $
satisfying (\ref{S6E4}) we obtain the existence of numbers $\xi_{k},\ k=1,2,3$
satisfying $0=\xi_{0}<\xi_{1}<\xi_{2}<\xi_{3}<\xi_{0}+T\left(  E\left(
\xi_{0}\right)  \right)  =\xi_{4}$ as well as
\begin{equation}
\frac{d\sigma\left(  \xi_{k}\right)  }{d\xi}=0\text{ , }k=0,1,2,3,4\,.
\label{A1E2}%
\end{equation}
Moreover
$\left(  U_{0}\left(  \xi_{0}\right)  -1\right)     =\left(  U_{0}\left(
\xi_{3}\right)  -1\right)  =\left(  U_{0}\left(  \xi_{4}\right)  -1\right)
=0 $, $
\sqrt{U_{0}\left(  \xi_{1}\right)  -1}    =\omega_{0}\left(  \xi_{1}\right)
=-\omega_{0}\left(  \xi_{2}\right)$
and
\begin{equation}
\frac{d\sigma\left(  \xi\right)  }{d\xi}<0\ \ \text{if\ \ }\xi\in\left(
\xi_{0},\xi_{1}\right)  \cup\left(  \xi_{2},\xi_{3}\right)   \ ,\ \ \frac
{d\sigma\left(  \xi\right)  }{d\xi}>0\ \ \text{if\ \ }\xi\in\left(  \xi
_{1},\xi_{2}\right)  \cup\left(  \xi_{3},\xi_{4}\right) \,. \label{A1E3}%
\end{equation}

The value of the total change of the energy during each cycle, namely
$\sigma\left(  \xi_{4}\right)  ,$ has been computed in (\ref{S6E7}),
(\ref{S6E8}) and approximated asymptotically as $E\rightarrow\infty$ (cf.
(\ref{S6E8a})). Notice that $\sigma\left(  \xi_{4}\right)  =\Phi\left(
E\right)  $, assuming that $E\left(  \xi_{0}\right)  =E.$

Due to (\ref{A1E2}), (\ref{A1E3}) 
 we need to compute $\sigma\left(  \xi_{1}\right)  , \sigma\left(
\xi_{2}\right)  , \sigma\left(  \xi_{3}\right)  $  in order to estimate the range
of variation of $\sigma$. We then have
\begin{equation}
\min\left\{  \sigma\left(  \xi_{1}\right)  ,\sigma\left(  \xi_{3}\right)
\right\}  \leq\sigma\left(  \xi\right)  \leq\max\left\{  \sigma\left(  \xi
_{2}\right)  ,\sigma\left(  \xi_{4}\right)  \right\}  \label{A1E3a}%
\end{equation}
where we use that $\sigma\left(  \xi_{4}\right)  >\sigma\left(  \xi
_{1}\right)  .$ In order to estimate these quantities we define an auxiliary
quantity $\Omega_{0}$ via $E = \Omega_0^2/2$ that measures the maximum value of $\omega_{0}$ for a
given value of $E$.
It will be more convenient to compute $\sigma\left(
\xi\right)  $ using $\omega_{0}$ as independent variable. To this end we
define $U_{-}\left(  \omega_{0};\Omega_{0}\right)  <1<U_{+}\left(  \omega
_{0};\Omega_{0}\right)  $ by means of the roots of the equation
\begin{equation}
\left(  U-1\right)  -\log\left(  U\right)  +\frac{\omega^{2}}{2}=\frac
{\Omega_{0}^{2}}{2}\,. \label{U2E3}%
\end{equation}

Taking into account (\ref{S6E3}) we can use $\omega_{0}$ as variable of
integration instead of $s$ in order to compute $\sigma\left(  \xi\right)  $
for $\xi\in\left(  0,\xi_{3}\right)  .$ Then, with $\omega_k:=
\omega_{0}\left(  \xi_{k}\right) $ we obtain
\begin{equation}
\sigma\left(  \xi_{k}\right)  =\sigma_{k}\left(  \Omega_{0}\right)
=\int_{\omega_{k}}^{\Omega_{0}}\left[  \left(  U_{+}\left(  \omega_{0}%
;\Omega_{0}\right)  -1\right)  -\left(  \omega_{0}\right)  ^{2}\right]
d\omega_{0}\ \ ,\ \ k=1,2,3\,, \label{A1E4}%
\end{equation}
where the integration is made along the curve defined by (\ref{U2E3}).
Notice that we use $\omega_{0}\left(  \xi_{3}\right)  =-\Omega_{0}.$

In order to approximate $\sigma\left(  \xi_{1}\right)  ,\ \sigma\left(
\xi_{2}\right)  ,\ \sigma\left(  \xi_{3}\right)  $ for small $\Omega_{0}$ we
use Taylor in (\ref{U2E3}) to obtain the approximation
\begin{equation}
\left(  U_{\pm}\left(  \omega_{0};\Omega_{0}\right)  -1\right)  ^{2}%
+\omega_{0}^{2}=\Omega_{0}^{2} \label{A1E5}%
\end{equation}
that combined with (\ref{A1E4}) yields
\[
\sigma\left(  \xi_{k}\right)  =\sigma_{k}\left(  \Omega_{0}\right)
=\int_{\omega_{k}}^{\Omega_{0}}\left[  \sqrt{\Omega_{0}^{2}-\omega_{0}^{2}%
}-\left(  \omega_{0}\right)  ^{2}\right]  d\omega_{0}%
\ \mbox{ as }\Omega_{0}\rightarrow0 \,.
\]
On the other hand, we can approximate $\omega_{1},\ \omega_{2}$ by 
$\sqrt{\Omega_{0}^{2}-\omega_{k}^{2}}=\left(  \omega_{k}\right)  ^{2}$ for $k=1,2$ 
whence
$\omega_{k}\sim\pm\Omega_{0}$ as $\Omega_{0}\rightarrow
0$ for $k=1,2$
and
\begin{equation}
\sigma\left(  \xi_{1}\right)  =o\left(  \Omega_{0}^{2}\right)  \ \ ,\ \ \sigma
\left(  \xi_{2}\right)  \sim\frac{\pi}{2}\Omega_{0}^{2}\ \ ,\ \ \ \sigma
\left(  \xi_{3}\right)  \sim\frac{\pi}{2}\Omega_{0}^{2}\text{\ \ as\ \ }%
\Omega_{0}\rightarrow0\,. \label{A1E6}%
\end{equation}

All these quantities must be compared with $\sigma\left(  \xi_{4}\right)  $
that can be computed using (\ref{S6E8}). The approximation (\ref{A1E5}) then
yields%
\begin{equation}
\sigma\left(  \xi_{4}\right)  \sim\pi\Omega_{0}^{2}\text{ \ as\ \ }\Omega
_{0}\rightarrow0 \,.\label{A1E7}%
\end{equation}

Combining (\ref{A1E6}), (\ref{A1E7}), as well as (\ref{A1E3a}) it follows
that
\[
o\left(  \Omega_{0}^{2}\right)  \leq\sigma\left(  \xi\right)  \leq\pi
\Omega_{0}^{2}\ \ ,\ \ \xi\in\left(  0,\xi_{4}\right)  \ \ \text{and }%
\Omega_{0}\rightarrow0^{+}\,.
\]

We can obtain a similar estimate for large values of $\Omega_{0}.$ Using
(\ref{U2E3}) we  obtain
\[
U_{+}\left(  \bar{\omega}_{0};\Omega_{0}\right)  -1\sim\frac{1}{2}\left(
\Omega_{0}^{2}-\bar{\omega}_{0}^{2}\right)  \ \ \text{as\ \ }\Omega
_{0}\rightarrow\infty
\]
that is valid as long as $\left(
U-1\right)  $ is large.

Since the region where $\left(  U-1\right)  $ is of order one gives a small
contribution to the integrals, we obtain the  approximations
\[
\sigma\left(  \xi_{1}\right)   
= - \Omega_0^3 \frac{\sqrt{3}}{9} = - \sigma (\xi_2)\,, \qquad \sigma(\xi_3)
= o(\Omega_0^3)
\qquad \mbox{ as } \Omega_{0}\rightarrow\infty. 
\]
Combining this with (\ref{S6E8a}) we obtain that
\[
-C\Omega_{0}^{3}\leq\sigma\left(  \xi\right)  \leq C\Omega_{0}^{3}%
\ \ ,\ \ \xi\in\left(  0,\xi_{4}\right)  \ \ \text{and }\Omega_{0}%
\rightarrow\infty
\]
for some $C>0$ independent of $\Omega_{0}.$ It then follows that the change of
energy during each cycle can be estimated by the final change. Therefore, this
justifies the adiabatic approximation.

\bigskip
\textbf{Acknowledgment:} BN and JJLV gratefully acknowledge the warm
atmosphere at the Isaac Newton Institute for Mathematical Sciences where part
of this work was done during the program on PDE in Kinetic Theories. This work
was also supported by the EPSRC Science and Innovation award to the Oxford
Centre for Nonlinear PDE (EP/E035027/1) and through the DGES Grant MTM2007-61755.

\end{document}